\documentclass[12pt]{amsart}
\usepackage{latexsym,amssymb,amsmath,amsthm}
\usepackage{comment}
\usepackage{float}
\usepackage{graphicx}
\newtheorem{thm}{Theorem}

\newtheorem{propo}[thm]{Proposition}
\theoremstyle{definition}

\newtheorem{observation}{Observation}
\def\dim{\rm{dim}}
\def\E{\rm{E}}

\title{Natural orbital networks}
\author{Oliver Knill}
\date{November 24, 2013}
\address{
        Department of Mathematics \\
        Harvard University \\
        Cambridge, MA, 02138
        }

\begin{document}
\maketitle
\begin{abstract}
Given a finite set $T$ of maps on a finite ring $R$,
we look at the finite simple graph $G=(V,E)$ with vertex set $V=R$ and edge set
$E=\{ (a,b) \; | \; \exists t \in T, \; b=t(a), \; {\rm and} \; b \neq a \}$.
An example is when $R=Z_n$ and $T$ consists of a finite set of quadratic maps
$T_i(x)=x^2+a_i$. Graphs defined like that have a surprisingly rich structure.
This holds especially in an algebraic setup when $T$ is generated by polynomials 
on $Z_n$. The characteristic path length $\mu$ and the mean
clustering coefficient $\nu$ are interlinked by global-local quantity 
$\lambda=-\mu/\log(\nu)$ which often appears to have a limit for $n \to \infty$  like for
two quadratic maps on a finite field $Z_p$. We see that for one quadratic map $x^2+a$, 
the probability to have connectedness goes to zero and for two quadratic maps, the 
probability goes to $1$, for three different quadratic maps $x^2+a,x^2+b,x^2+c$ on $Z_p$, 
we always appear to get a connected graph for all primes.
\end{abstract}

\begin{center}
In Memory of Oscar Lanford III.
\end{center}

\section{Introduction}

The interest in applied graph theory has exploded during the last decade
\cite{Goyal,BornholdtSchuster,GoodmanORourke,CohenHavlin,newman2010,nbw2006,
WassermanFaust,Jackson,SmallWorld, BallobasKozmaMiklo,ibe,vansteen,newman2010,Meester,Easley,shen}.
The field is interdisciplinary and has relations with other subjects like computer, biological
and social sciences and has been made accessible to a larger audience in books like
\cite{SixDegrees,Buchanan, Linked,Sync,Connected}. \\

We look here at generalizations of Cayley graphs which can be used to model networks  
having statistical features similar to known random networks
but which have an appearance of real networks. We have mentioned some connections with 
elementary number theory already in \cite{KnillMiniatures}. 
The graphs are constructed using {\bf transformation monoids} analogously as 
{\bf Cayley graphs} are defined by a group acting on a group. The deterministic construction 
produces graphs which resemble empirical networks seen in synapses, chemical or biological 
networks, the web or social networks, and especially in peer to peer networks. \\

A motivation for our construction is the iteration of quadratic maps in the complex
plane which features an amazing variety of complexity and mathematical content.
If the field of complex numbers $C$ is replaced by a finite ring $R$ like $Z_n$, we 
deal with arithmetic dynamical systems in a number theoretical setup.
This is natural because a computer simulates differential equations on a finite set.
It is still not understood well, why and how finite approximations of
continuum systems share so many features with the continuum. Examples of such
investigations are \cite{vivaldi,Rannou,lanford98}.
The fact that polynomial maps produce arithmetic chaos is exploited by pseudo 
random number generators like the Pollard rho method for integer factorization \cite{Riesel}. \\

A second motivation is the construction of deterministic realistic networks.
Important quantities for networks are the characteristic path length $\mu$ and
the mean cluster coefficient $\nu$. Both are averages: the first one is the average of 
$\mu(x)$, where $\mu(x)$ is the average over all distances $d(x,y)$ with $y \in V$,
the second is the average of the fraction of the number of edges in
the sphere graph $S(x)$ of a vertex $x$ divided by the number of edges of the complete graph
with vertices in $S(x)$. Nothing changes when replace $\mu$ with the {\bf median path length} 
and $\nu$ with the {\bf global cluster coefficient} which is the frequency of 
oriented triangles $3 v_2/t_2$ within all oriented connected vertex triples. \\

A third point of view is more algebraic. As we learned from \cite{Cameron2011},
orbital networks relate to {\bf automata}, edge colored directed graphs with possible self loops
and multiple loops. An automaton can encode the monoid acting on the finite set.
But the graph of the monoid $T$ is obtained by connecting two vertices $x,y$ if there
is no element $T$ with $T(x)=T(y)$. This graph is the null graph if $M$ is synchronizing
(there is $f$ which has has a single point as an image) and the complete group if $M$ is a
permutation subgroup of the full permutation group \cite{Cameron}.
An example of a problem in that field is 
the task to compute the probability that two random endofunctions generate
a synchronizing monoid. This indicates that also the mathematics of graphs generated by 
finitely many transformations can be tricky. We became also aware of \cite{Steinberg}, 
who developed a theory of finite transformation monoids. In that terminology, 
the graph generated by $T$ is called an {\bf orbital digraph}. 
This prompted us to address the finite simple graphs under consideration 
as {\bf orbital networks}. \\

That the subject has some number theoretical flavour has been indicated already \cite{KnillMiniatures}.
It shows that elementary number theory matters when trying to understand connectivity properties of
the graphs. We would not be surprised to see many other connections. \\

Modeling graphs algebraically could have practical values. Recall that in computer vision, various
algorithms are known to represent objects. The highest entropy version is to give a triangularization
of a solid or to give a bitmap of a picture. Low entropy realizations on the other hand store the object
using mathematical equations like inequalities or polynomials in several variables. It is an AI task
to generate low entropy descriptions. This can mean to produce vector graphics representations of a given
bitmap, or to use B\'ezier curves and Nurb surfaces to build objects. 
Fractal encoding algorithms have been used to encode and compress pictures: iterated function
systems encode similar parts of the picture \cite{Barnsley}. Similarly, an application of orbital networks 
could be to realize parts of networks with low entropy descriptions in such a way that relevant 
statistical properties agree with the actual network. \\

In the context of computer science, monoids are important to describe languages as {\bf synthactic monoids}.
An analogue of Cayley's theorem tells that every group is a subgroup of a transformation group, every finite 
monoid can be realized as a transformation monoid on a finite set $X$.
When seen from an information theoretical point of view, 
an orbital network describes a language, where the vertices are the alphabet and 
the transformations describe the rules. Every finite path in the graph now describes a possible
word in the language. \\

Finally, there is a geometric point of view, which actually is our main interest. Graphs share
remarkably many parallels with Riemannian manifolds. Key results on Riemannian manifolds or
more general metric spaces with cohomology have direct analogues for finite simple graphs. While 
functionals like the Hilbert action on Riemannian manifolds are difficult to study, finite simple 
graphs provide a laboratory to do geometry, where one can experiment with a modest amount of effort. 
We will look elsewhere at the relation of various functionals on graphs, like the Euler characteristic, 
characteristic path length, Hilbert action. It turns out that some of the notions known to graph theory
only can be pushed to Riemannian manifolds. Functionals studied in graph theory can thus be studied also
in Riemannian geometry.  \\

\begin{figure}
\scalebox{0.2}{\includegraphics{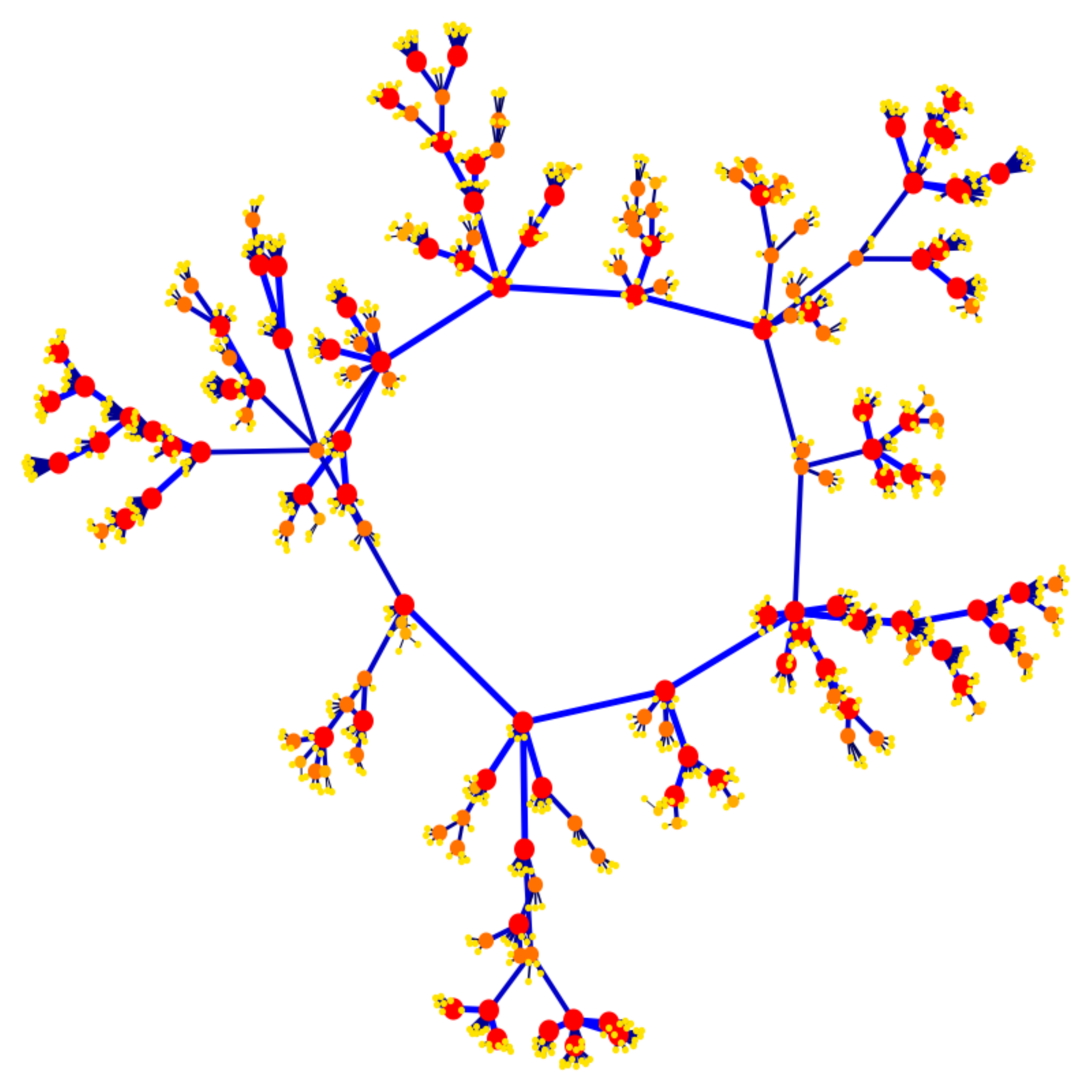}}
\scalebox{0.2}{\includegraphics{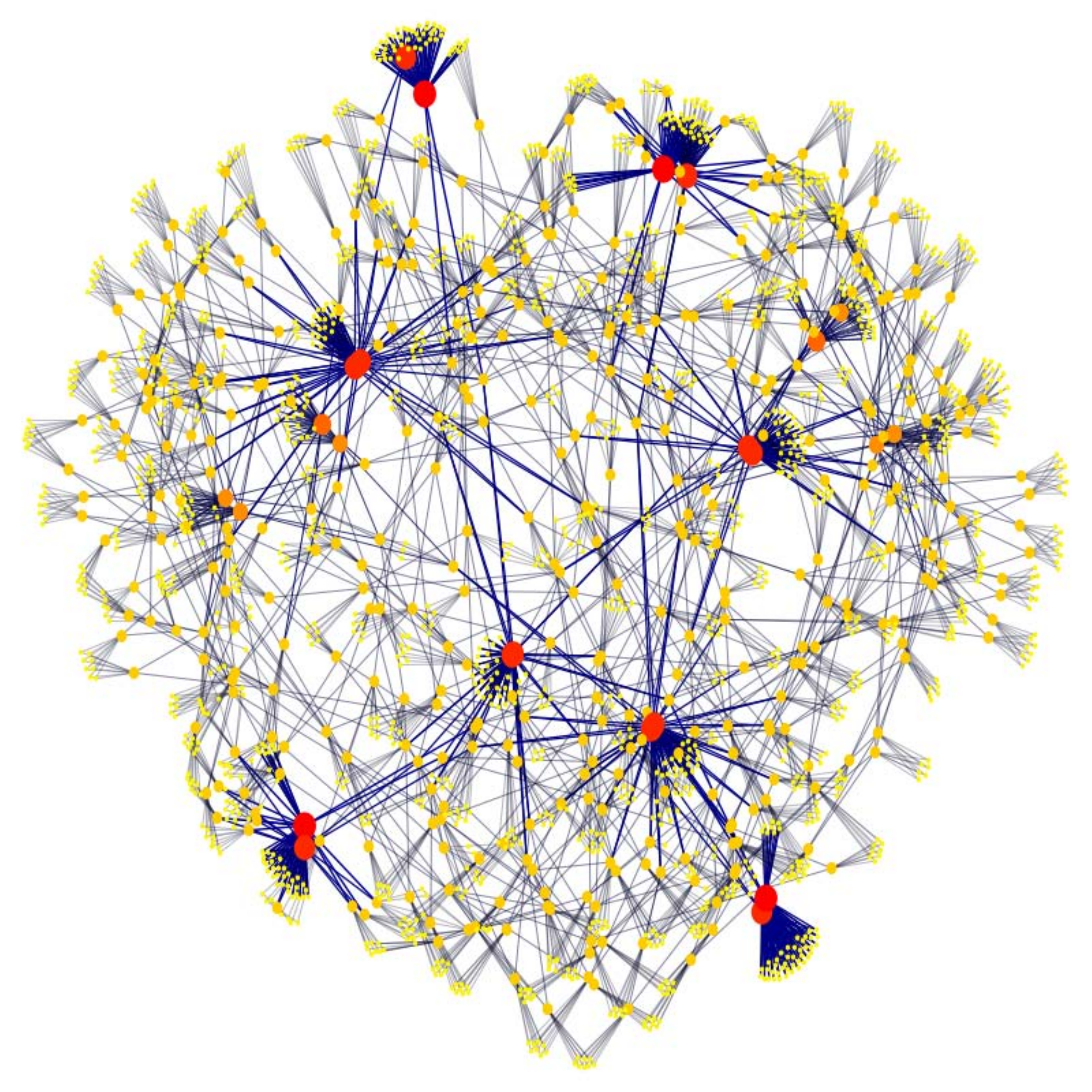}}
\caption{
The first graph $G_1$ is obtained from a single quadratic map $f(x)=x^2+226$ on $Z_{1001}$. It
has diameter $14$, average degree $\delta=2$, characteristic path length $\mu=8.6$ and $\nu=0$. 
It resembles a typical friendship graph \cite{Jackson}.
The second graph is generated by $f(x)=x^2+1,g(x)=x^2+2$ on $Z_{2000}$. It has diameter $9$,
$\delta=3.99,\mu=5.8,\nu=0.0024$ and $\lambda=0.964$. 
}
\end{figure}

\begin{figure}
\scalebox{0.2}{\includegraphics{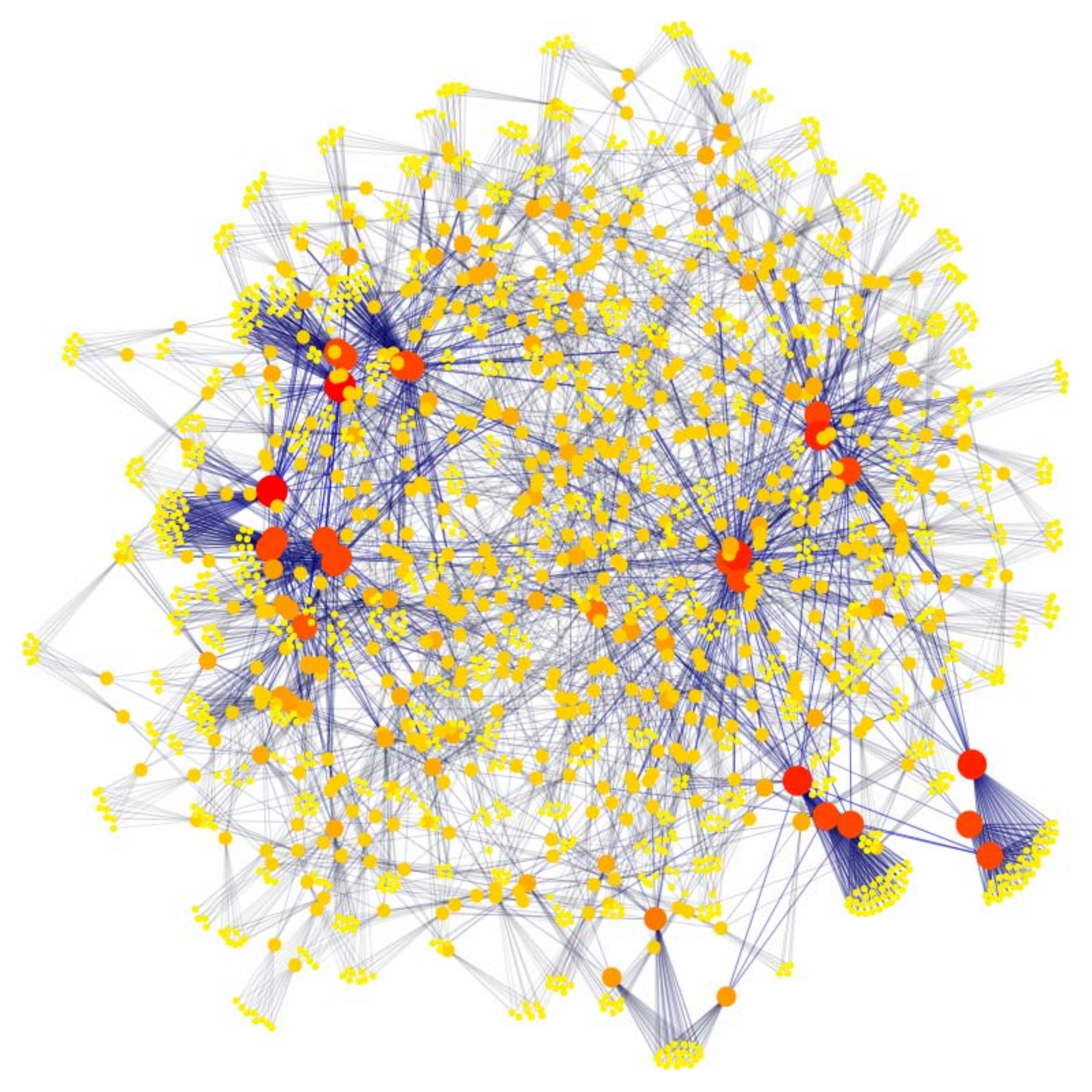}}
\scalebox{0.2}{\includegraphics{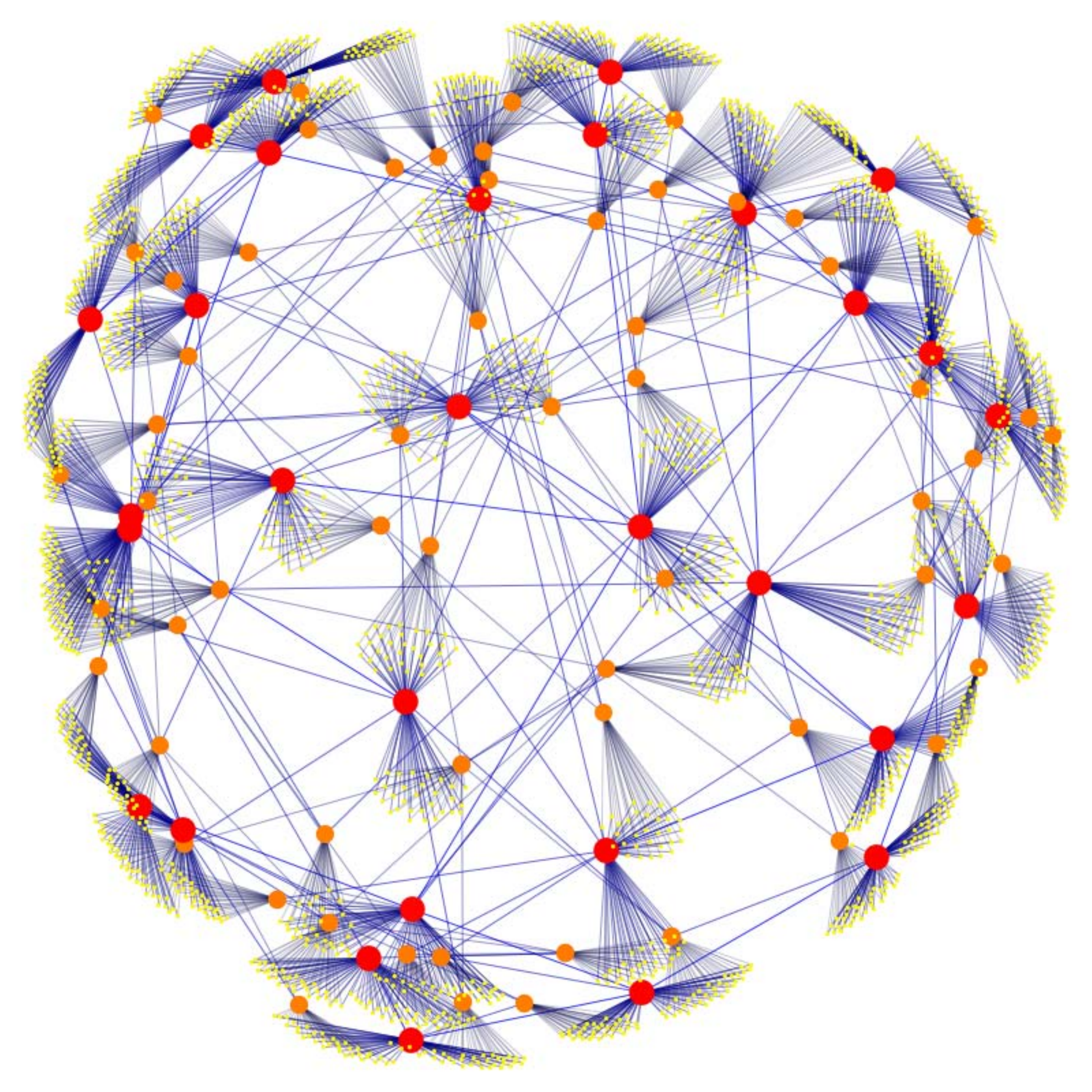}}
\caption{
We see the graph  generated by $f(x)=x^2+1,g(x)=x^2+31,h(x)=x^2+51$ on $Z_{2000}$. It has
diameter $8$, average degree $6$, $\mu=4.7, \nu=0.0084$ and $\lambda=0.986$.
The second  graph is generated by $T(x)=2^x+11$ and $S(x)=3^x+5$ on $Z_{2002}$.
It has diameter $7$, $\mu=4.6, \nu =0.00097$ and $\lambda=0.66$.}
\end{figure}

\begin{figure}
\scalebox{0.12}{\includegraphics{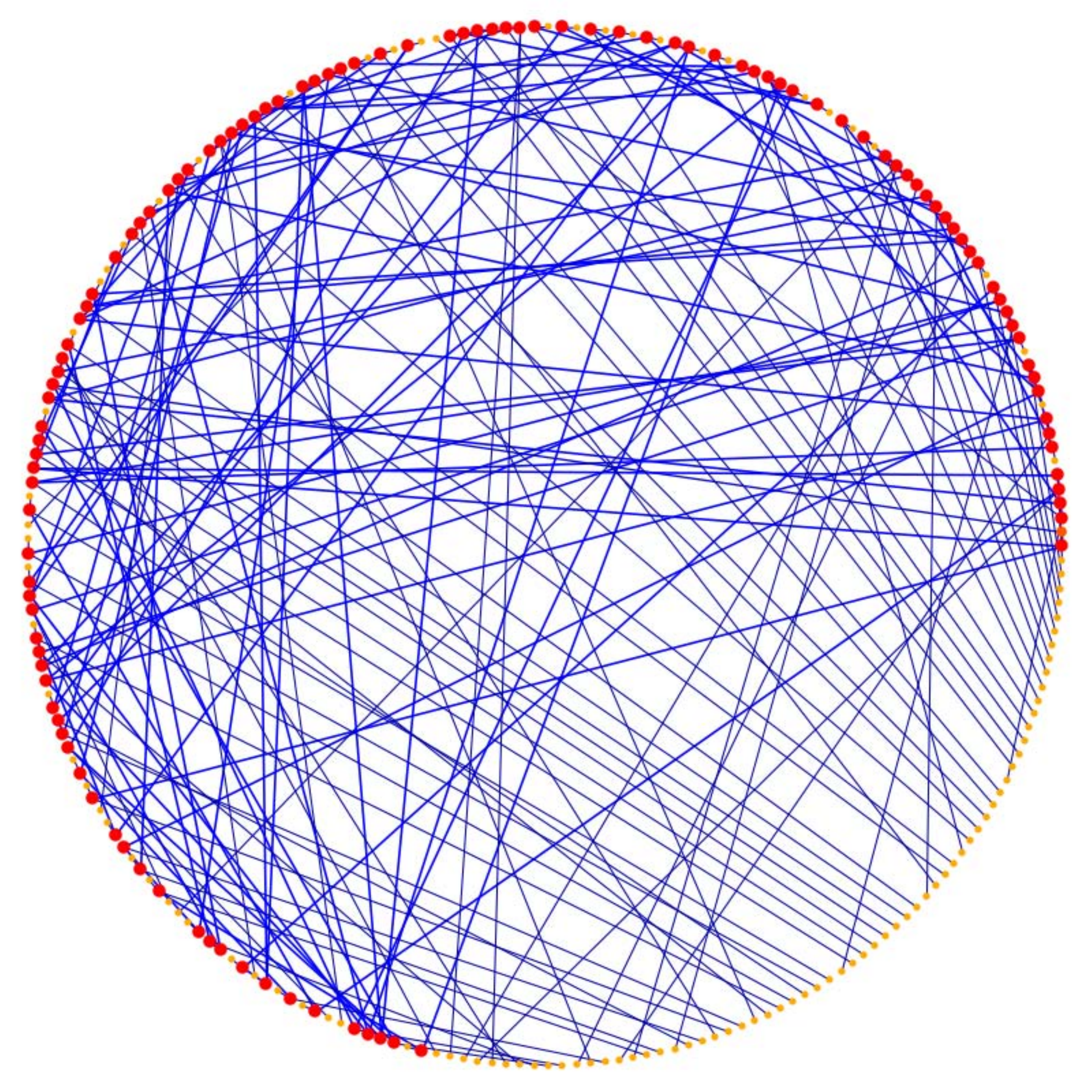}}
\scalebox{0.12}{\includegraphics{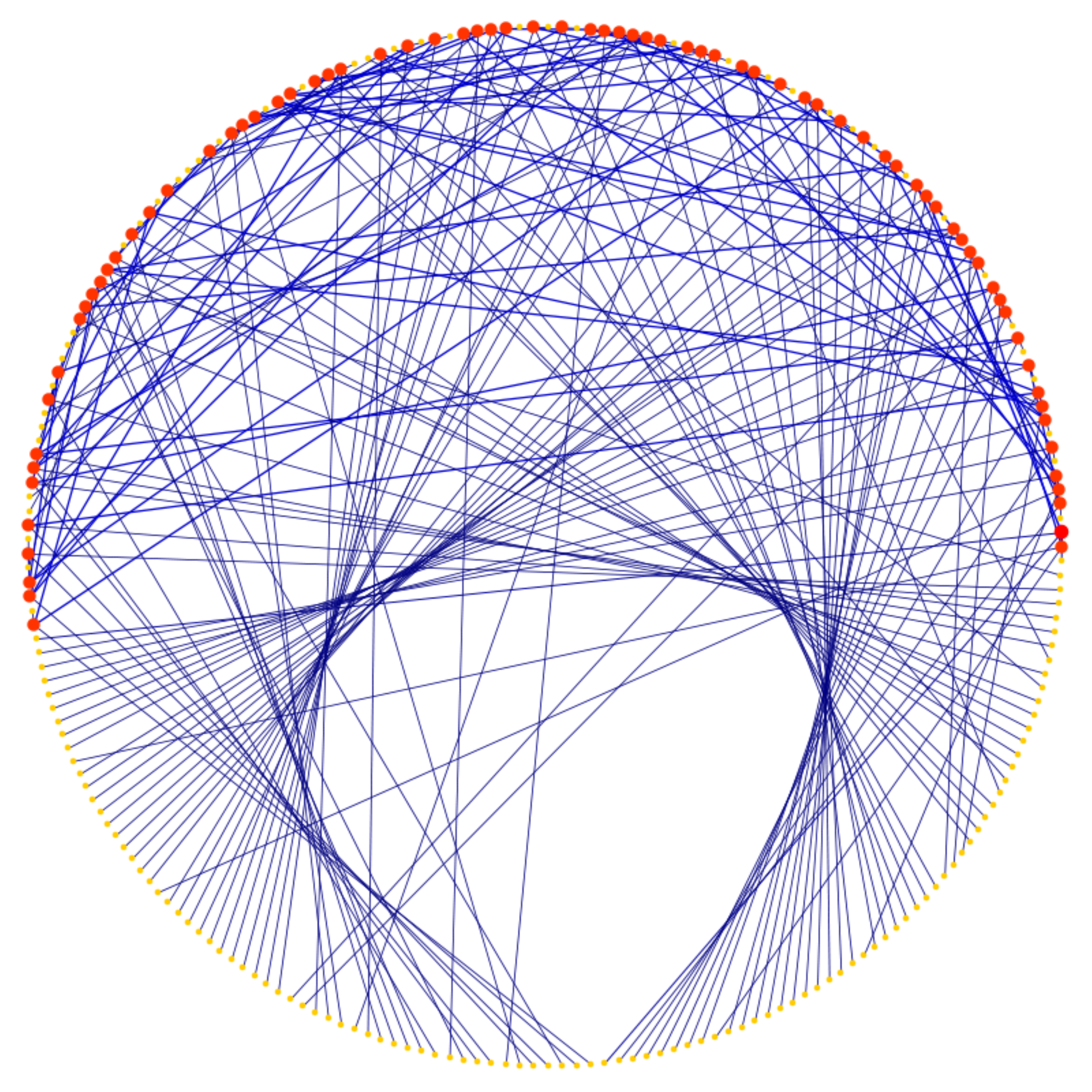}}
\caption{
The graph on $Z_{229}$ generated by $T(x)=x^2+57$ and a graph
on $Z_{229}$ generated by $T(x)=2^x+13$ are shown with a circular embedding. 
}
\end{figure}

{\bf Acknowledgement}. The graph construction described here emerged during a few meetings
in September and October 2013 with Montasser Ghachem who deserves equal credit 
for its discovery \cite{GK1}, but who decided  not to be a coauthor of this paper.

\section{Construction}

Given a fixed ring $R$ like $R=Z_n$ and a finite set $T$ of maps $R \to R$,
a digraph is obtained by taking the set of vertices $V=R$ and edges $(x,T_i(x))$, 
where $T_i$ are in $T$. By ignoring self loops, multiple connections as well as directions, we obtain a 
finite simple graph, which we call the {\bf dynamical graph} or {\bf orbital network} 
generated by the system $(R,T)$. The name has been used before and distinguishes
from other classes of dynamically generate graphs: either by 
random aggregation or by applying deterministic or random transformation rules to a 
given graph.  \\

\begin{figure}[H]
\scalebox{0.15}{\includegraphics{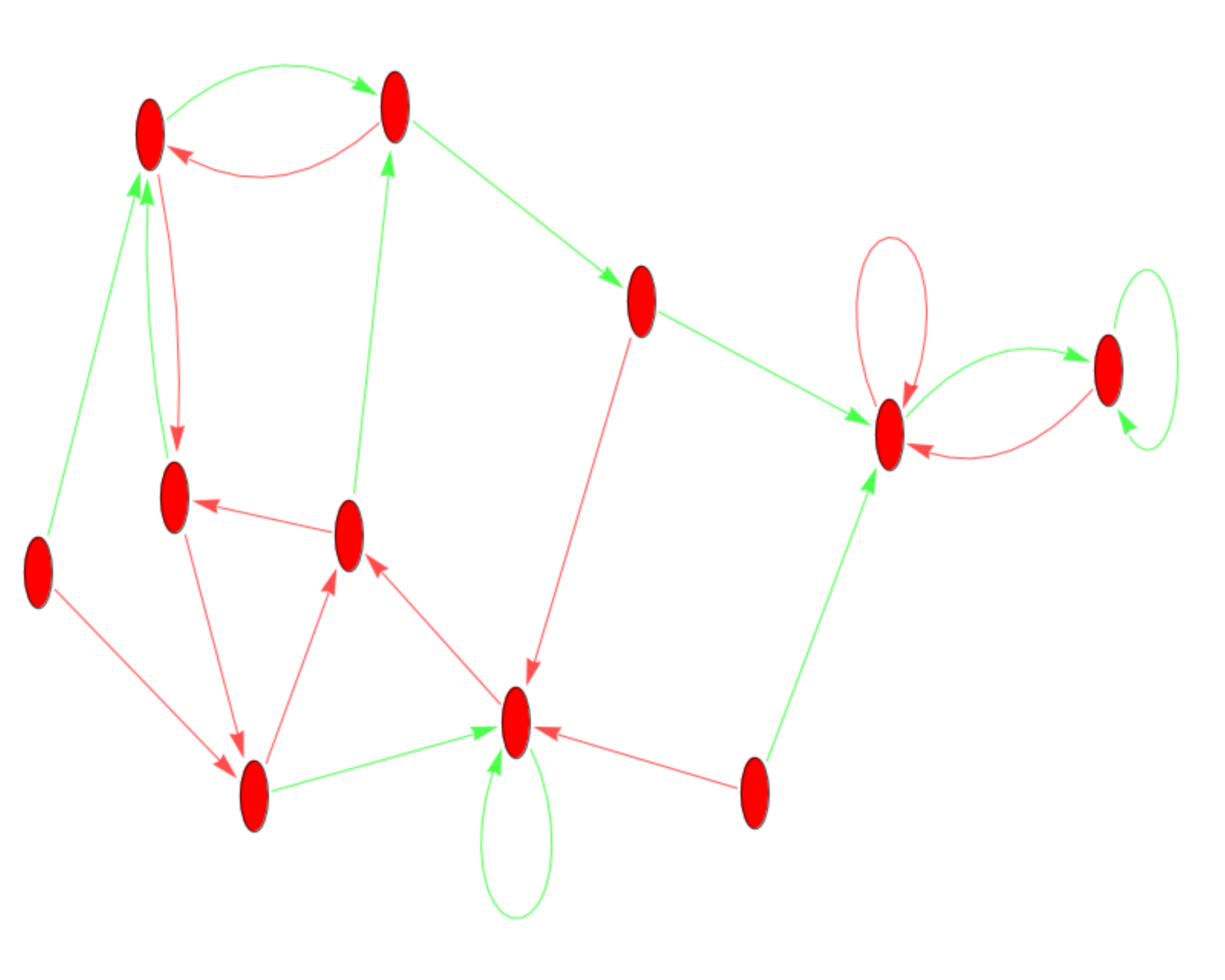}}
\scalebox{0.15}{\includegraphics{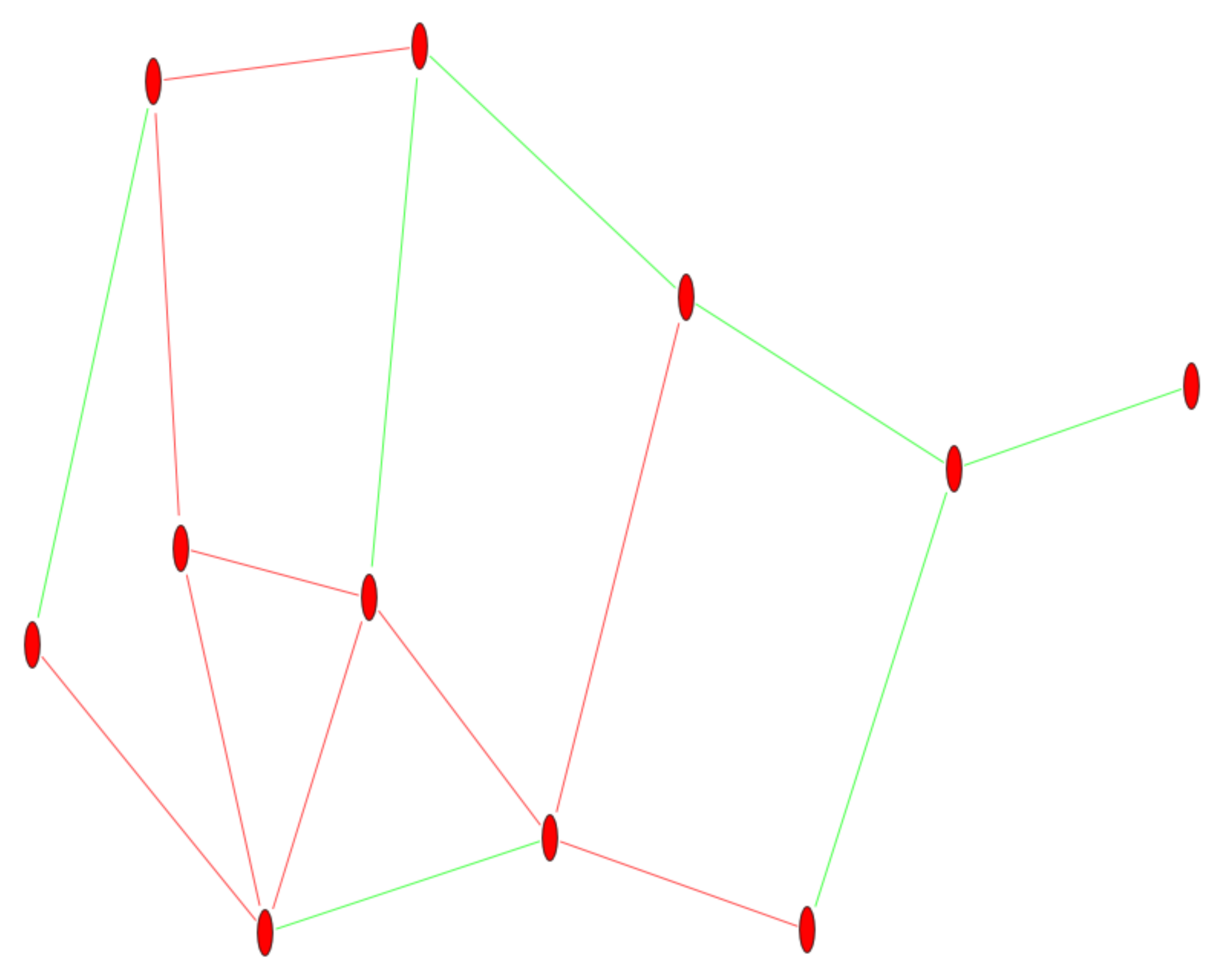}}
\caption{
The two maps $T_1(x)=x^2+3$ and $T_2(x)=x^2+2$ on $Z_{11}$ define
an {\bf orbital digraph} with edges $x \to T_i(x)$. Ignoring the direction, 
multiple connections and self loops we get a finite simple graph which we 
call an {\bf orbital network}.}
\end{figure}

Orbital networks have been used since a long time. If a group given by a finite set of 
generators $T$ acts on itself, the visualization is called the {\bf Cayley graph} of the group. It is
a directed graph but leaving away the direction produces a finite simple graph
which is often used as a visualization for the group. These graphs are by definition
vertex transitive and especially, the vertex degree is constant.
For one-dimensional dynamical systems given by maps defined on finite sets, the
graph visualization has been used at least since Collatz in 1928 \cite{Wirsching}.
It appears also in demonstrations \cite{WolframState1,WolframState2} or
work \cite{Kayama} on cellular automata. The graphs are
sometimes called {\bf automata networks} \cite{Robert}. 
The iteration digraph of the quadratic map has been studied in \cite{SomerKrizek}:
As has been noted first by Szalay in 1992, as has been a symmetry if $n=2$ modulo $8$ 
or $n=4$ modulo $8$.  Szalay also noted that the number of fixed points of $x \to x^2$ is 
$2^{\omega(n)}$ where $\omega(n)$ is the number of distinct primes dividing $n$. 
The miniature theorem on Fermat primes rediscovered in \cite{KnillMiniatures} has already been 
known to Szalay in 1992 and Rogers in 1996.
Somer and Krizek show that there exists a closed loop of length $t$ if and only if 
$t=ord_n(2)$ for some odd positive divisor $d$ of the Carmichael lambda function $\lambda(n)$
and where ${\rm ord}_n(g)$ is the multiplicative order of $g$ modulo $n$.  \\
 
When looking at classes of maps, we get a probability space of graphs.
For example, if we consider two quadratic maps $T_i(x)=x^2+c_i$ with $c_i \in R$. 
we get a probability space $R^2$ of parameters $(c_1,c_2)$.
With $k$ quadratic maps, the probability space $R^k$ presents itself. \\

An other example is to take the ring $R=Z_n \times Z_n$ and
to consider the H\'enon type maps $T_i(x,y) = (x^2+c_i-y,b_i x)$, where $c_i$ are parameters,
where again with $k$ maps the probability space is $R^k$. If $b_i=1$, the maps are invertible and
generate a subgroup of the permutation group. An interesting class of graphs
obtained with $T_i(x) = [x^{\alpha} + c_i]$ on $Z_n$, where $[x]$ is the floor function
will be studied separately.  We focus first on maps which are defined in an arithmetic way.
An other example is by replacing polynomials with exponential maps like with
$T_i(x) = 2^x+c_i$ on $Z_n$ or linear matrix maps like $T_i(x,y) = A_i x$ on $Z_n^k$ or
higher degree polynomial maps like $T_i(x,y) = a x^2 +b y^2 + c x y + c_i$.
An example of algebro-geometric type is to look at the ring $R= K[x]/I$, where $K[x]$ is
the polynomial ring over a finite field $K$ and $I$ an ideal, then consider the two maps
$T_1(f) = f^2+a, T_2(f) = f^2 +b$, where $a,b$ are fixed polynomials. \\

\begin{observation}
Every finite simple graph is an orbital network with $d$ generators,
where $d$ is the maximal degree of the graph. 
\end{observation}
\begin{proof}
If the graph has $n$ vertices, take
$R=Z_n$. Now label the elements in each sphere $S(x)$ with $y_1,\dots, y_{k(x)} \leq d$.
Now define $T_i(x)=y_i$ for $i \leq k(x)$ and $T_i(x)=x$ for $i>k(x)$.
\end{proof}

We often can use less generators. Of course our interest is not for general graphs,
but rather for cases, where $T_i$ are given by simple arithmetic formulas, where we have a natural
probability space of graphs. For example, if $T$ is generated by $d$ polynomials
$T_i$ on $Z_n$, then $Z_n^d$ equipped with a counting measure is a natural probability space. \\

An other natural example is if $T_i$ are random permutations on $Z_n$. This is a random model
but the statistics often mirrors what we see in the case of arithmetic transformations, only
that in the arithmetic case, the $n$ dependence can fluctuate, as it depends on the factorization 
of $n$.  \\

\begin{observation}
If the monoid generated by $T$ can be extended to a group,
then a connected orbital network is a factor of a Cayley graph.
\end{observation}
\begin{proof}
If the monoid generated by $T$ is a group $U$, then 
all $T_i$ are permutations. Take a vertex $x_0 \in V$. Since the graph is connected, $Ux_0$ is
the entire vertex set. The graph homomorphism $\phi(f) = f x_0$
shows that the graph $G$ is a factor of the Cayley graph of $G$ generated by $T$. 
\end{proof}

Remark. Embedding the monoid in a group using a
Grothendieck type construction is not always possible.

\section{Comparison with real networks}

\begin{figure}[H]
\scalebox{0.09}{\includegraphics{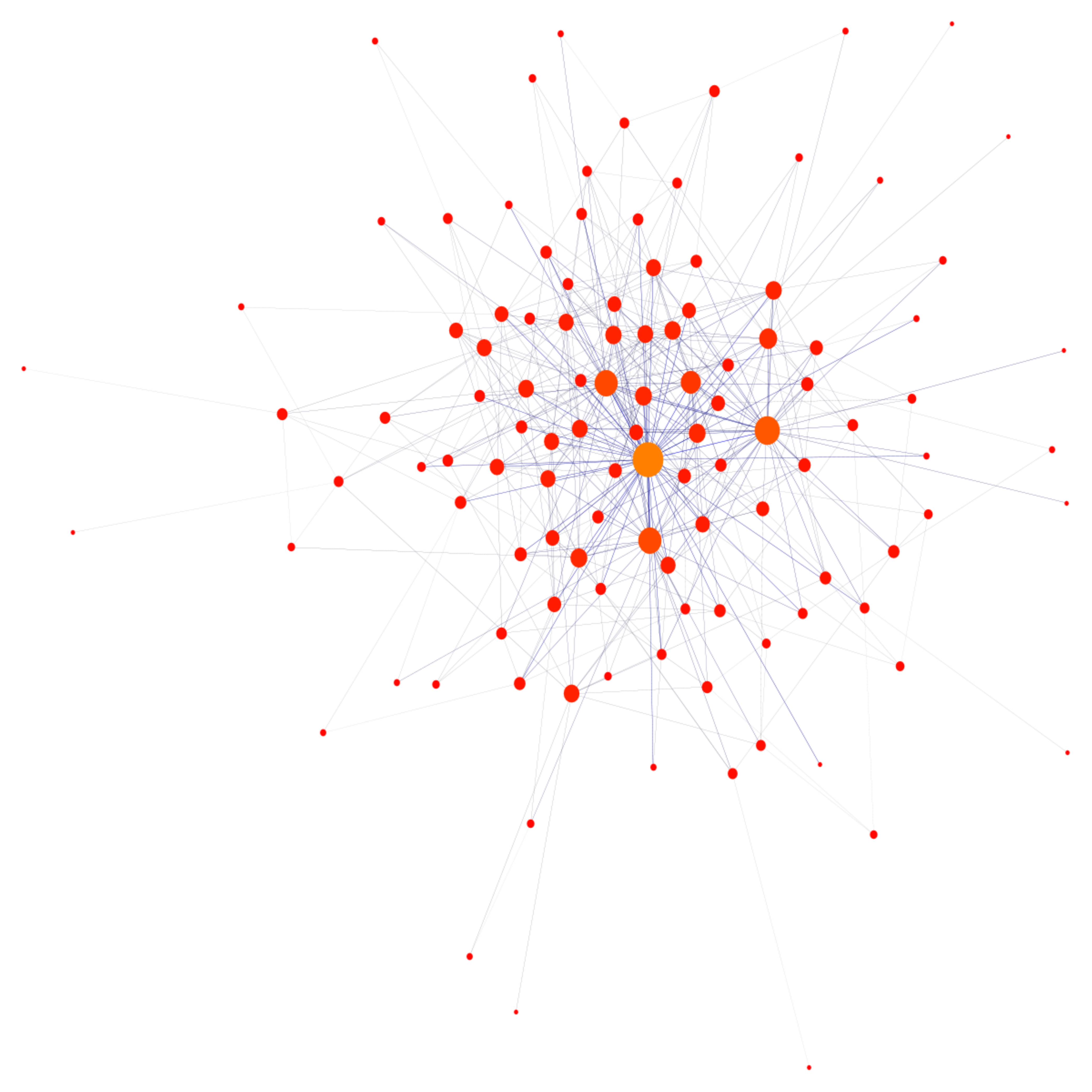}}
\scalebox{0.09}{\includegraphics{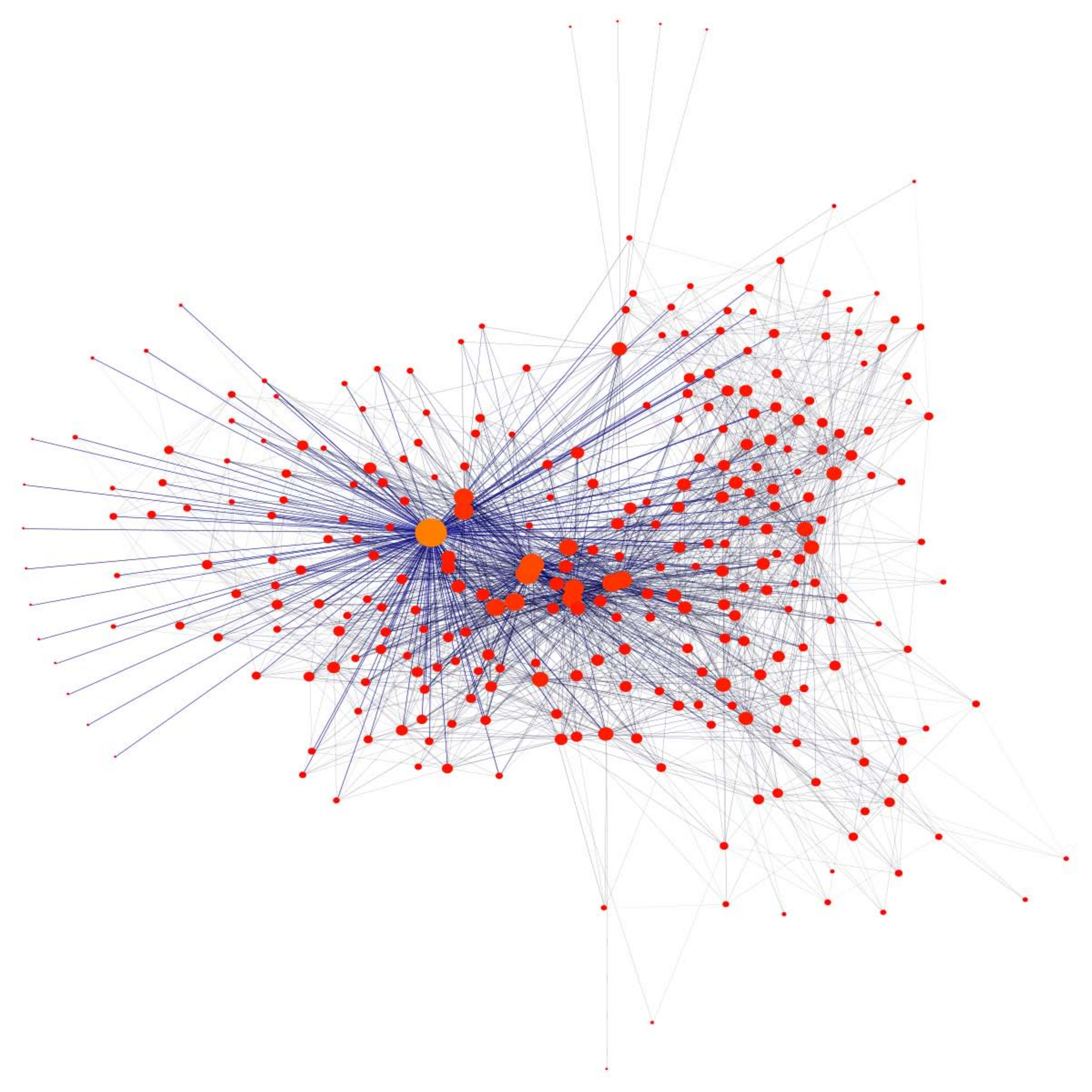}}
\scalebox{0.09}{\includegraphics{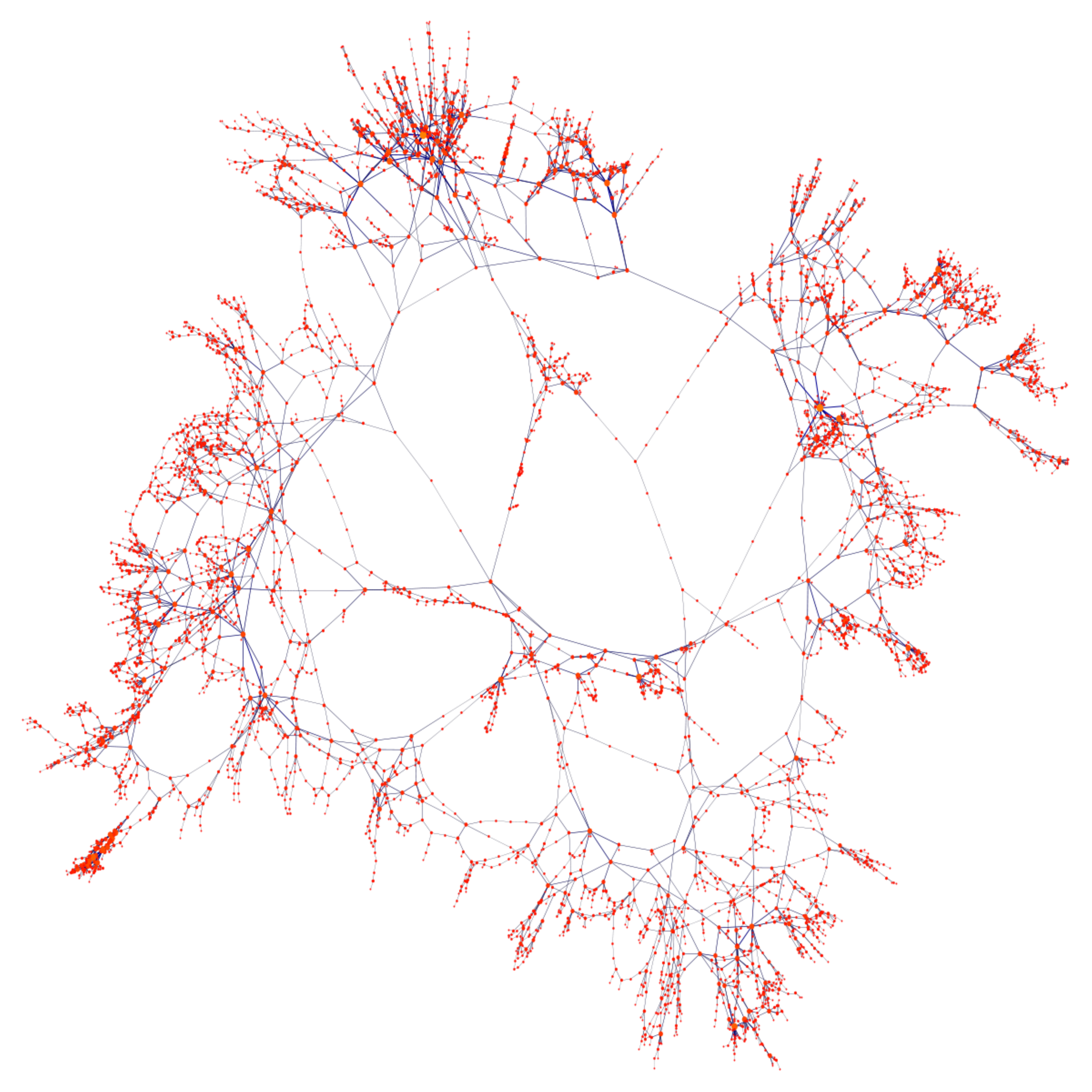}}
\caption{
Examples of publicly available networks.
The first is the word adjacency graph of common adjectives and nouns
in "David Copperfield" \cite{Newman2006},
the second is a neural network of Caenorhabditis Elegans 
\cite{WattsStrogatz}, the third graph shows 
the Western states power grid \cite{WattsStrogatz}. We have
$\lambda({\rm adjnoun}) = 1.36915$, 
$\lambda({\rm celegans}) = 1.43514$ and
$\lambda({\rm powergrid}) = 8.35961$, 
}
\end{figure}

\begin{figure}[H]
\scalebox{0.09}{\includegraphics{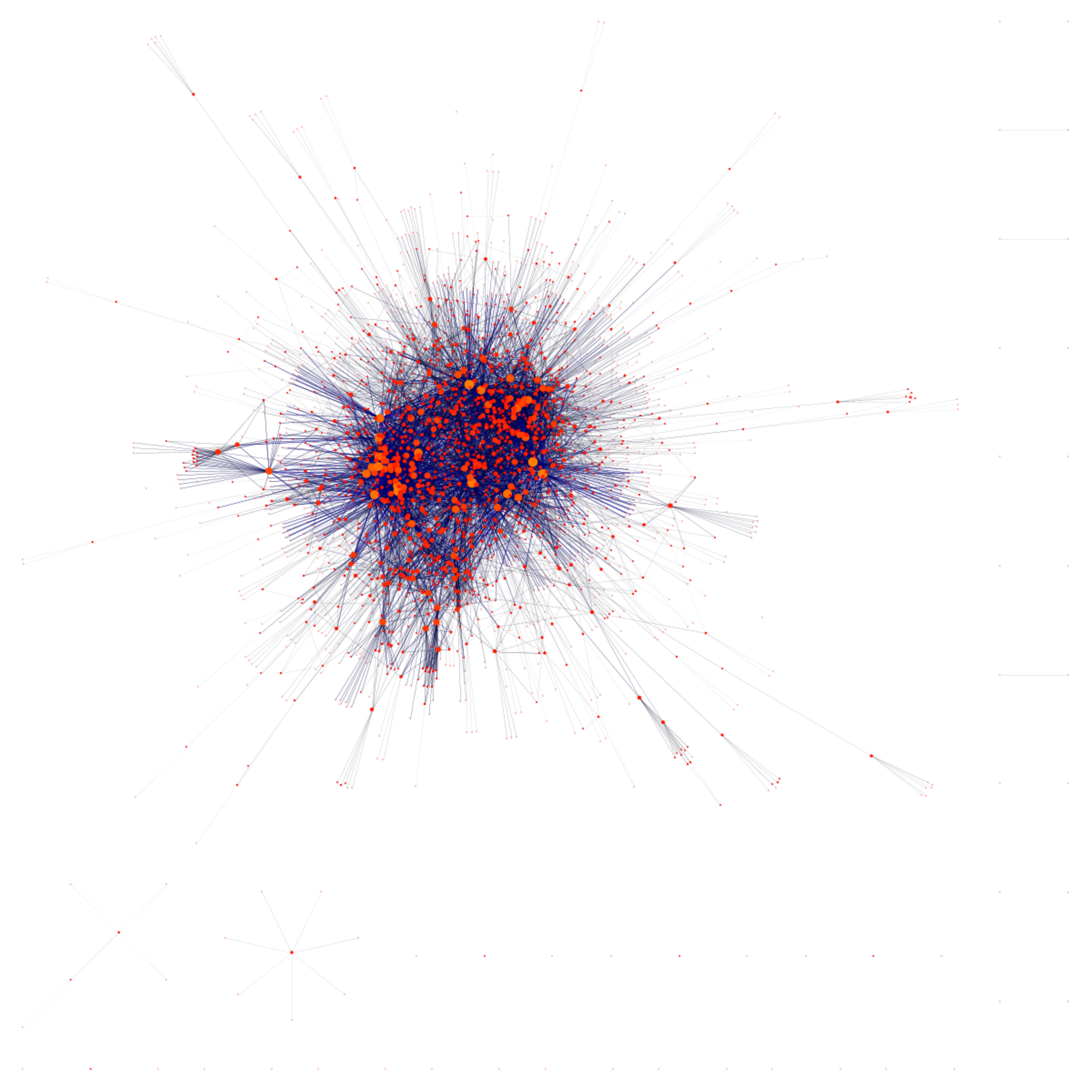}}
\scalebox{0.09}{\includegraphics{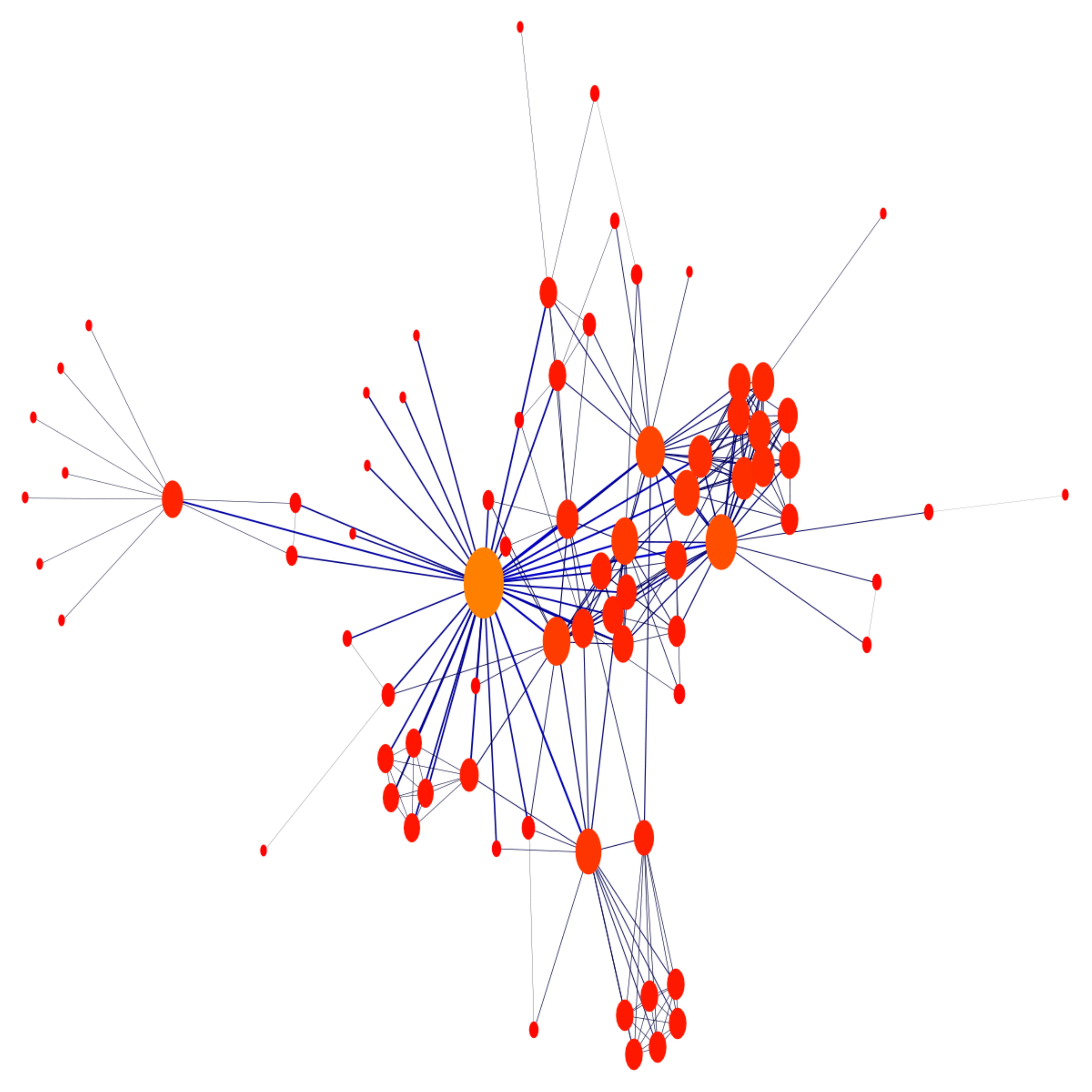}}
\scalebox{0.09}{\includegraphics{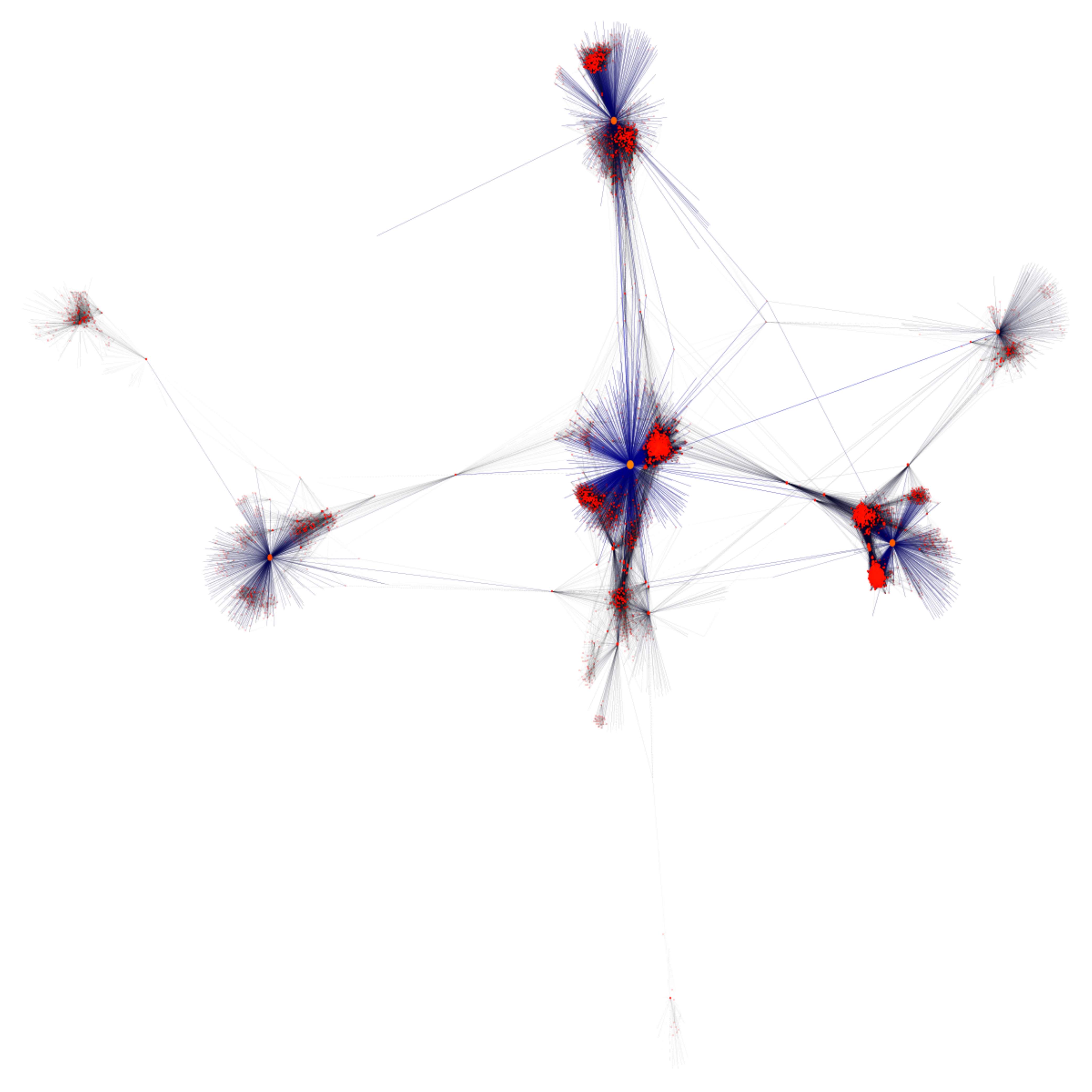}}
\caption{
The yeast protein-protein network \cite{WattsStrogatz},
the characters in the novel Les Miserables \cite{KnuthGraphBase,Gephi}
and finally part of facebook from the Stanford database. 
We measure the length-cluster coefficients:
$\lambda({\rm yeast}) = 1.82013$,
$\lambda({\rm miserables}) = 3.79865$ and
$\lambda({\rm facebook}) = 5.63298$,
}
\end{figure}

\begin{figure}
\scalebox{0.08}{\includegraphics{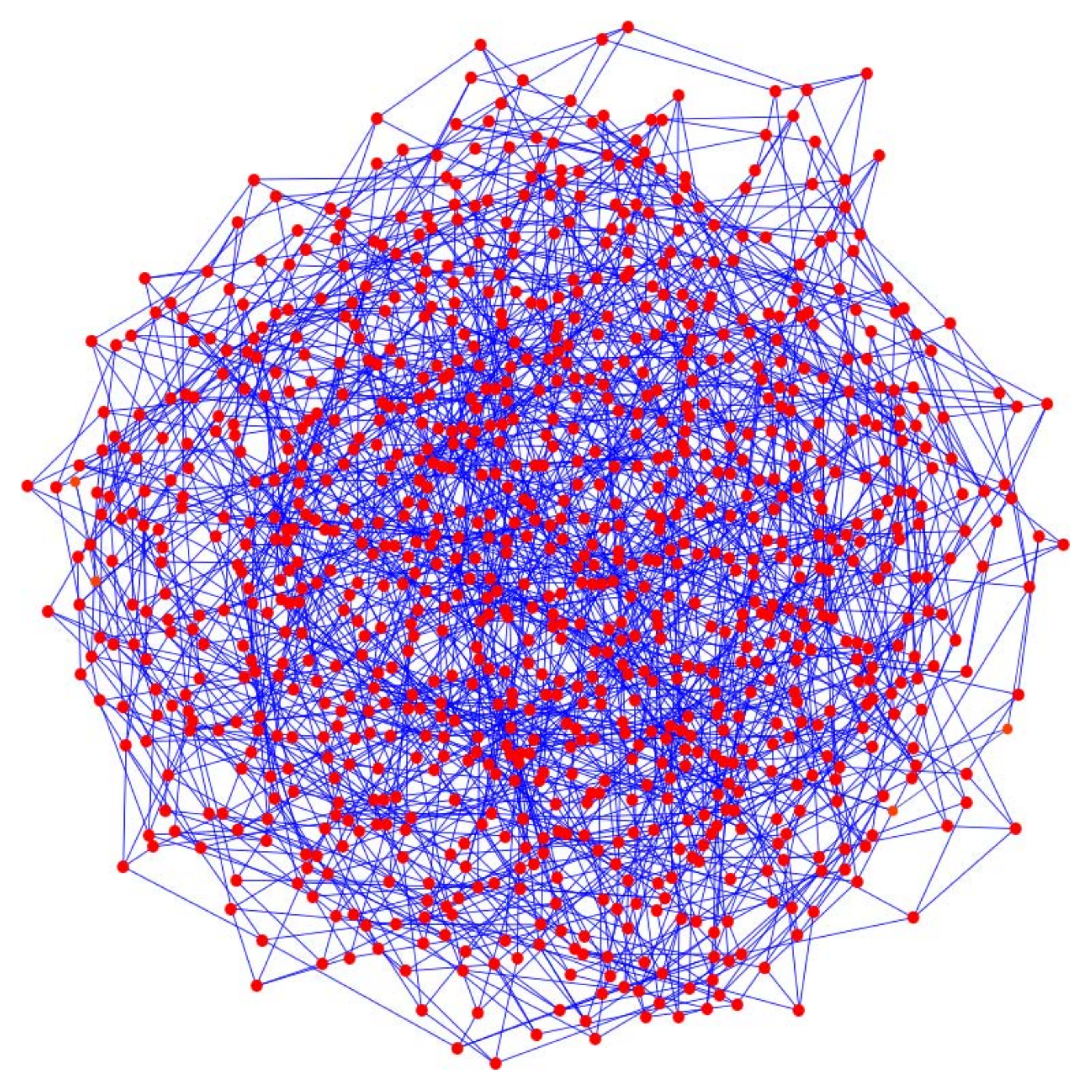}}
\scalebox{0.08}{\includegraphics{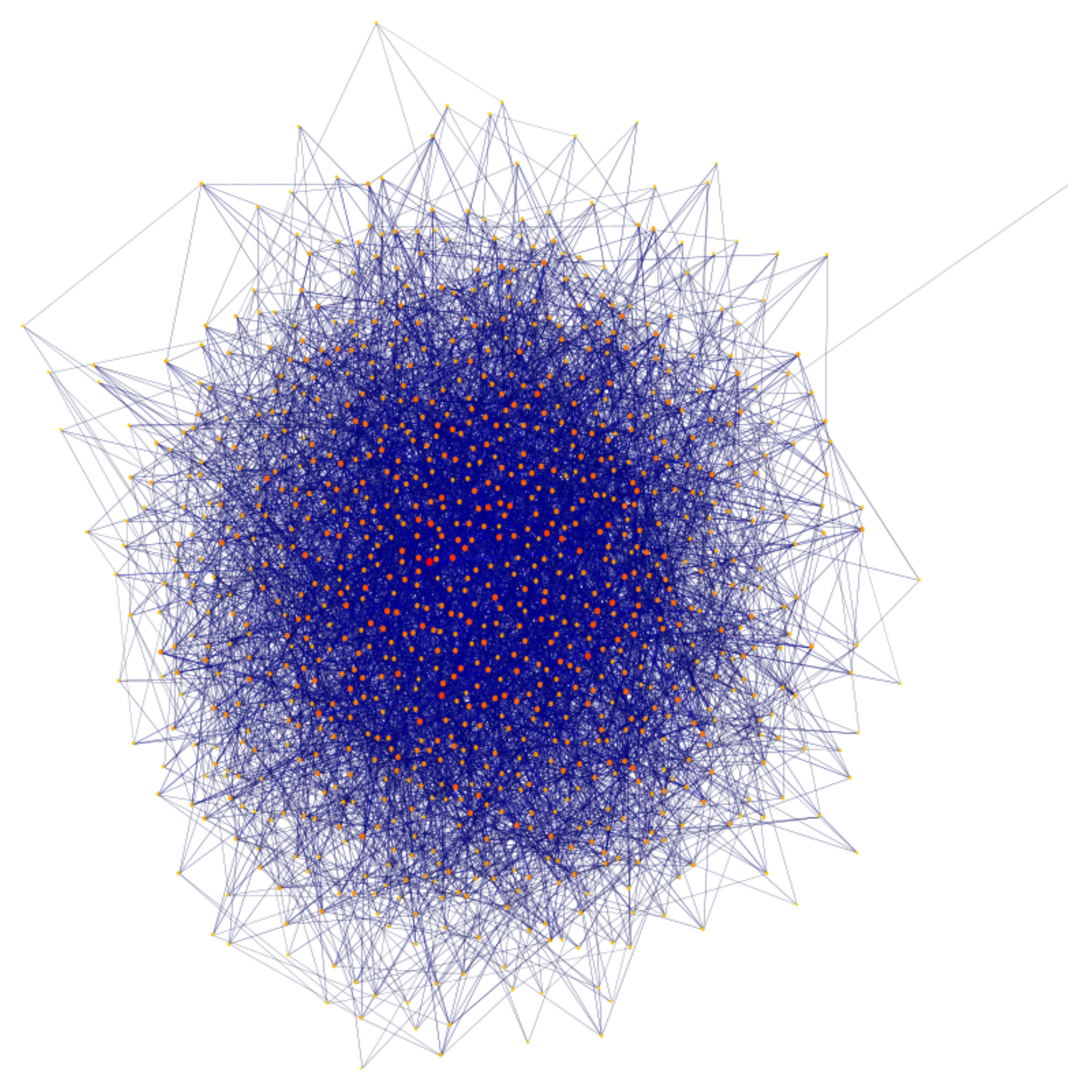}}
\scalebox{0.08}{\includegraphics{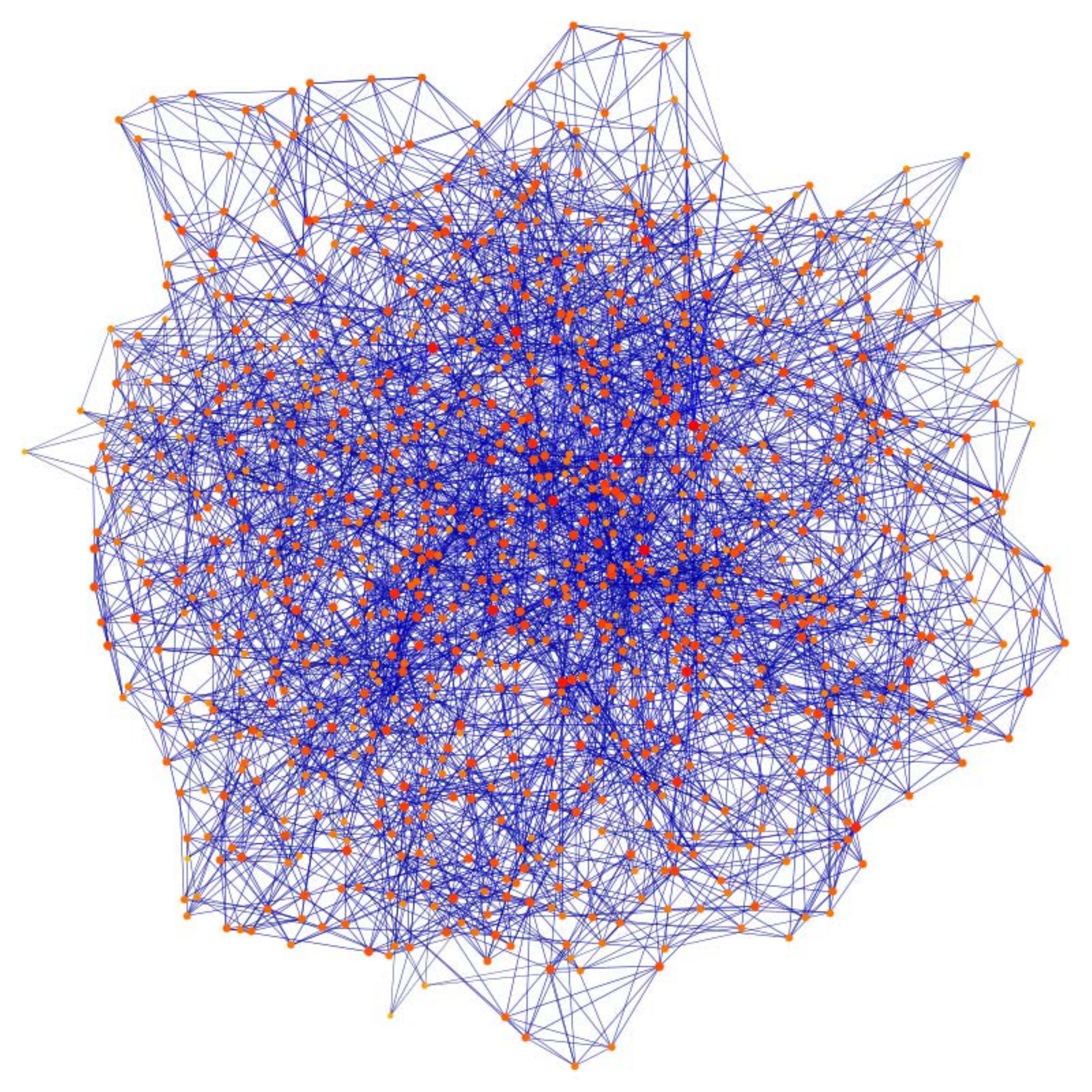}}
\scalebox{0.08}{\includegraphics{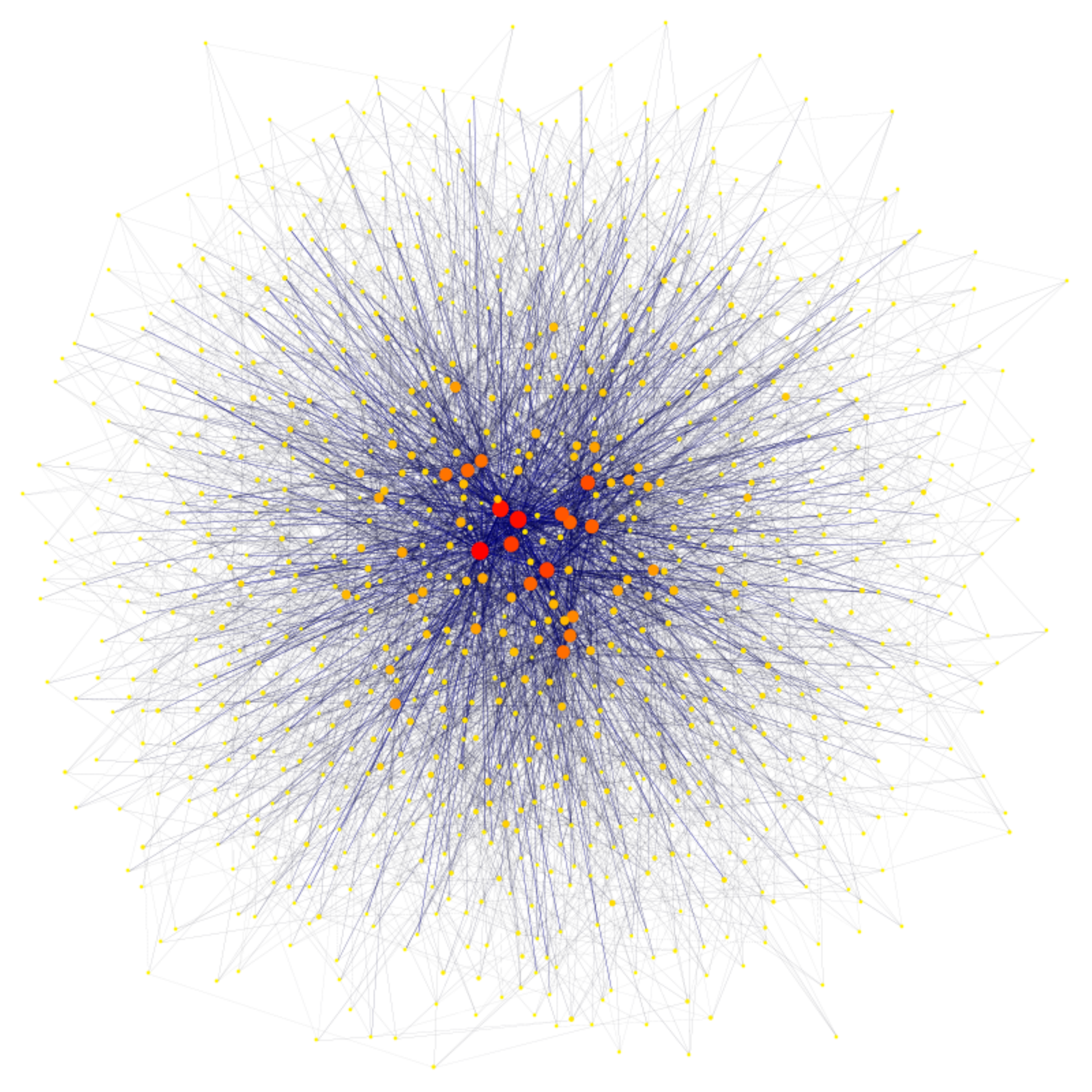}}
\caption{
Examples of random graphs with 1001 vertices. Their shape is in general more amorphic.
The first is a {\bf random permutation graph} $T,S$ on $Z_{1001}$ instead of arithmetic maps. 
Its average degree is very close to $4$, the characteristic path length is $5.6$, the global 
cluster coefficient $0.001$ and the length cluster coefficient is $0.81$.
The second figure shows a {\bf random Erdoes-Renyi graph} \cite{erdoesrenyi59} with 1001 vertices.
It has average degree 9.8, characteristic path length 3.3, global cluster coefficient
$0.0078$ and $\lambda=0.67$.
The third figure is a {\bf Watts-Strogatz graph}  \cite{WattsStrogatz} with 1001 nodes with an $8$-regular graph,
rewiring probability of $0.2$, characteristic path length $4.3$, global cluster
coefficient $0.32$ and $\lambda=3.9$. 
The fourth figure shows a {\bf random Barabasi-Albert graph} \cite{BarabasiAlbert}  with $k=4$. 
The average degree is close to $8$, the characteristic path length $3.2$ and the
global clustering coefficient $0.025$ and $\lambda=0.866$. }
\end{figure}

\section{Length-Cluster coefficient} 

Various graph quantities can be measured when studying graphs.
Even so we have a deterministic construction, we can look at graph 
quantities as {\bf random variables} over a probability space
of a class of generators $T$. Examples are the global degree average $2|E|/|V|$,
the Euler characteristic $\chi(G) = \sum_{k=0} (-1)^k c_k$, 
where $c_k$ is the number of $K_{k+1}$
subgraphs of $G$, the inductive dimension $\dim(G)$ \cite{elemente11},
the number of cliques $v_k(G)$ of dimension $k$ in $G$, the average degree $2c_1/c_0 = 2|E|/|V|$, 
the Betti numbers $b_k(G)$, the diameter ${\rm diam}(G)$, the characteristic path length $L$, 
the global clustering coefficient $C$ which is the average occupancy density in spheres
\cite{WattsStrogatz} or the variance of the degree distribution or
curvature $K(x) = \sum_{k=0}^{\infty} (-1)^k V_{k-1}(x)$, where 
$V_i(x)$ is the number of subgraphs $K_{k+1}$ in the unit sphere $S(x)$ of a vertex.  \\

\begin{figure}
\scalebox{0.12}{\includegraphics{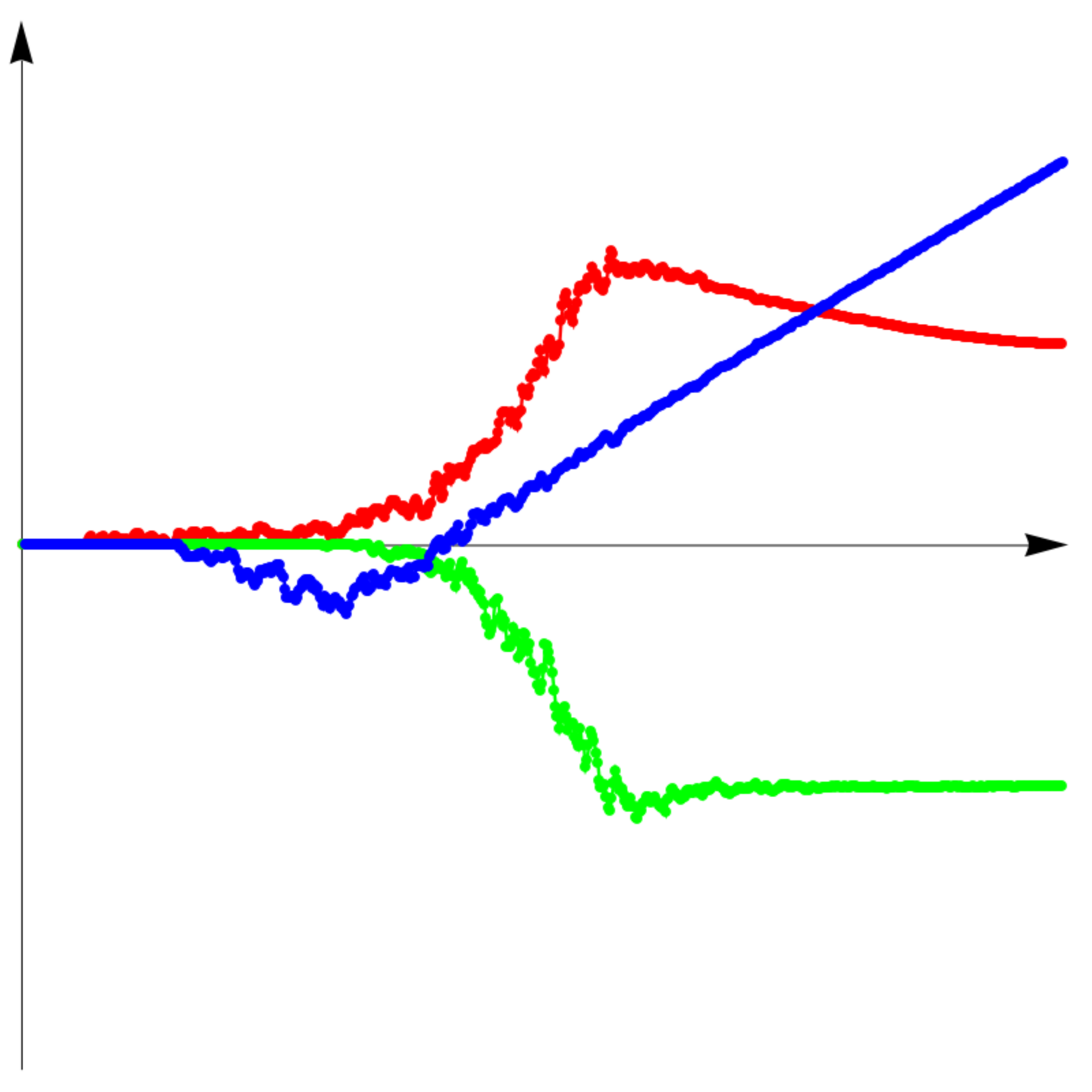}}
\scalebox{0.12}{\includegraphics{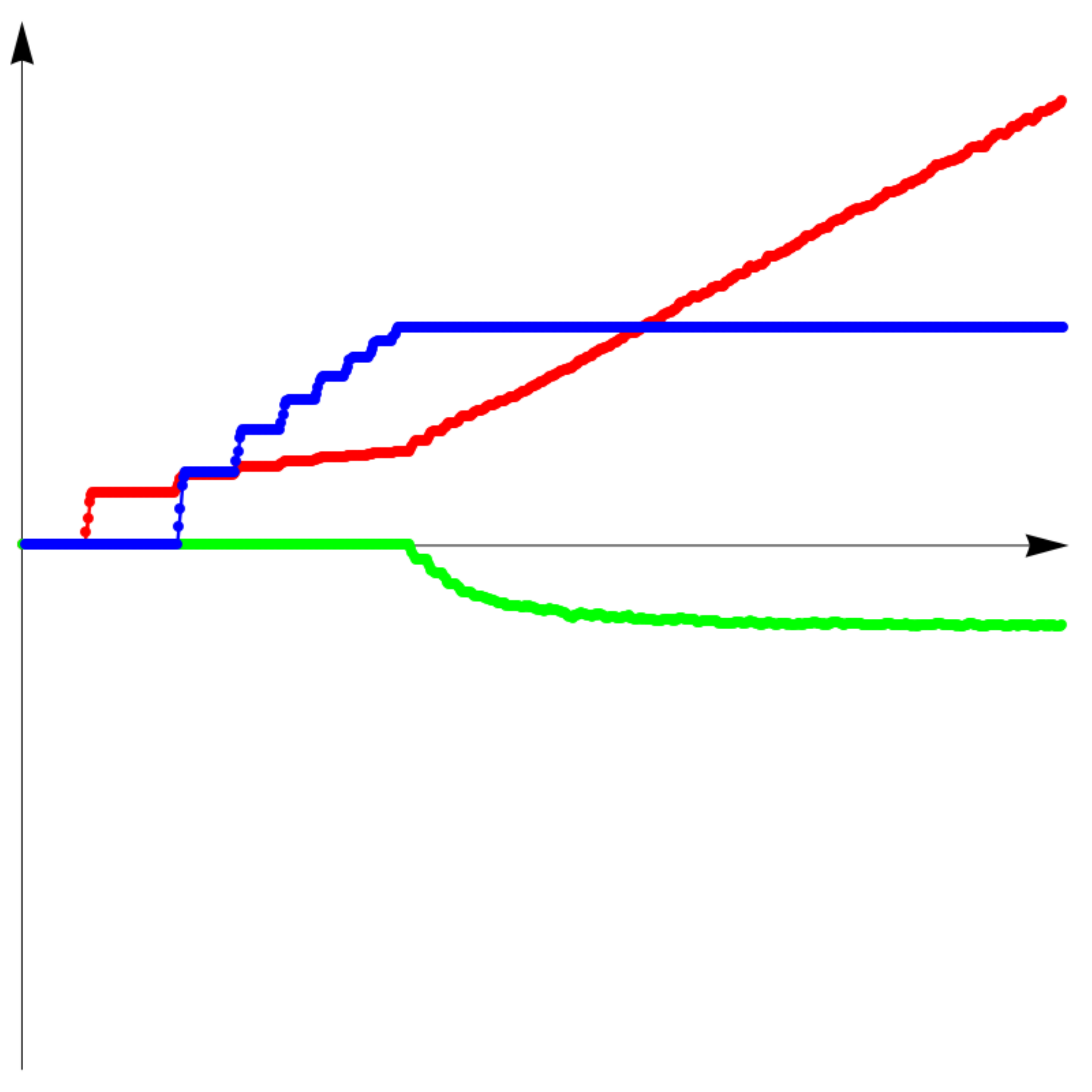}}
\scalebox{0.12}{\includegraphics{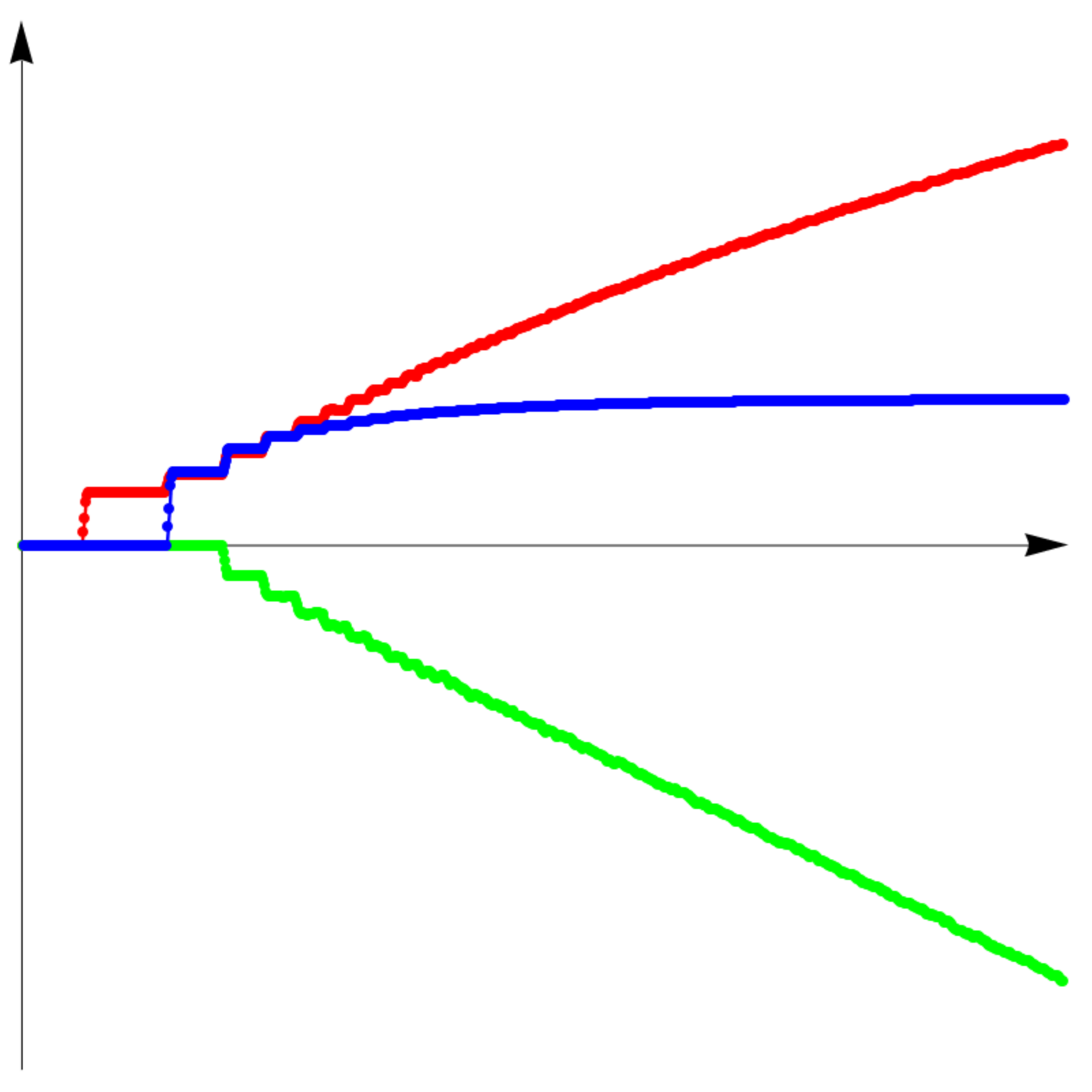}}
\scalebox{0.12}{\includegraphics{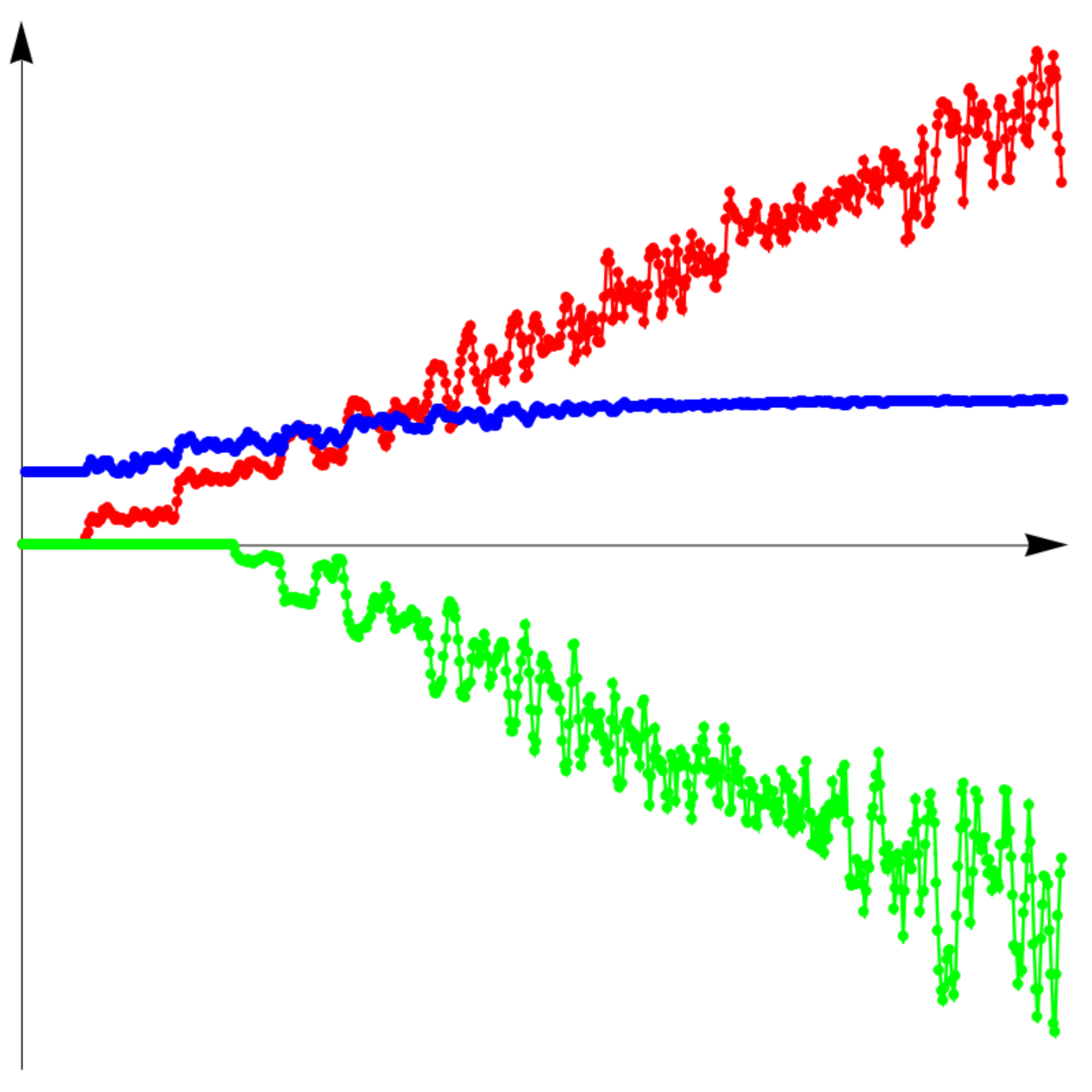}}
\scalebox{0.12}{\includegraphics{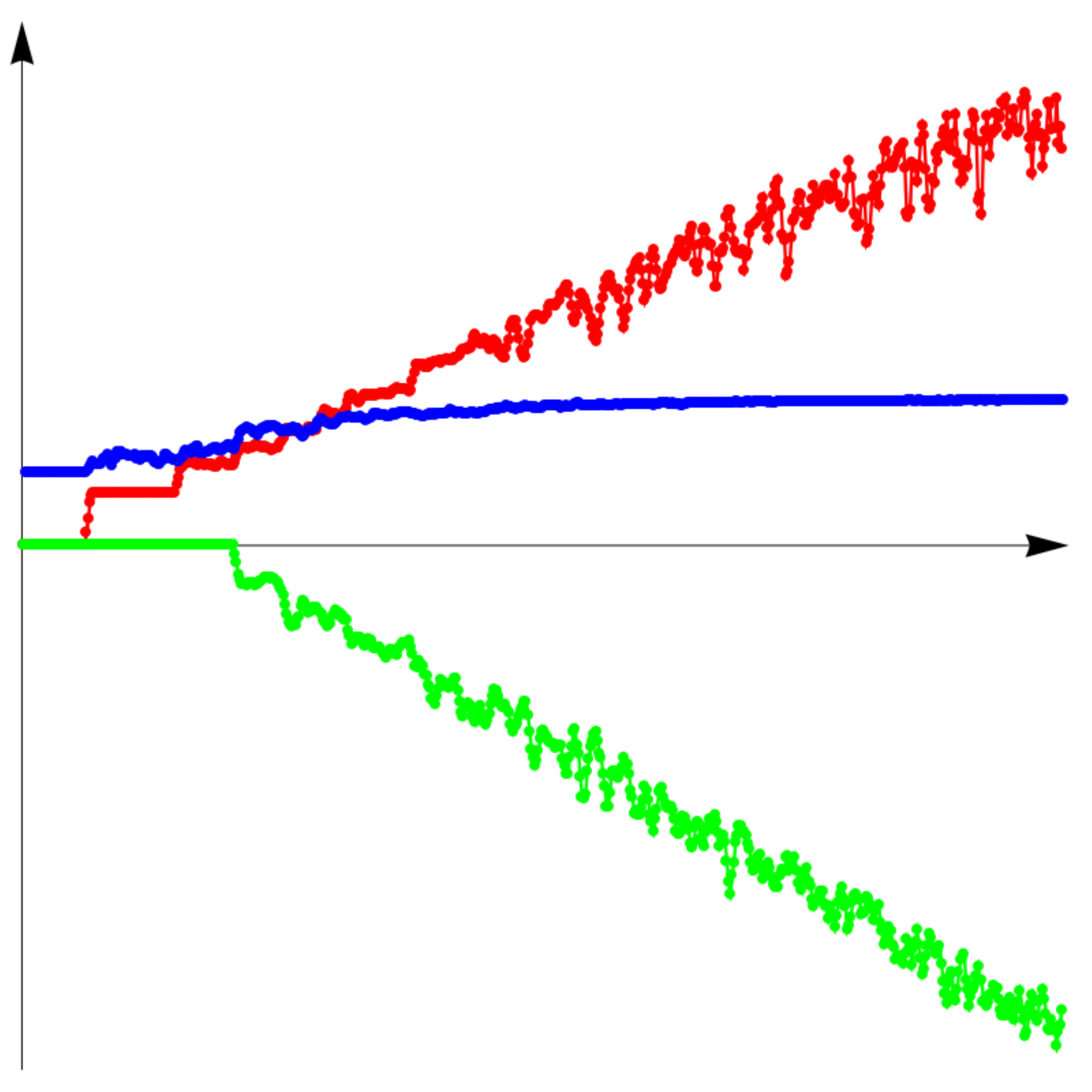}}
\scalebox{0.12}{\includegraphics{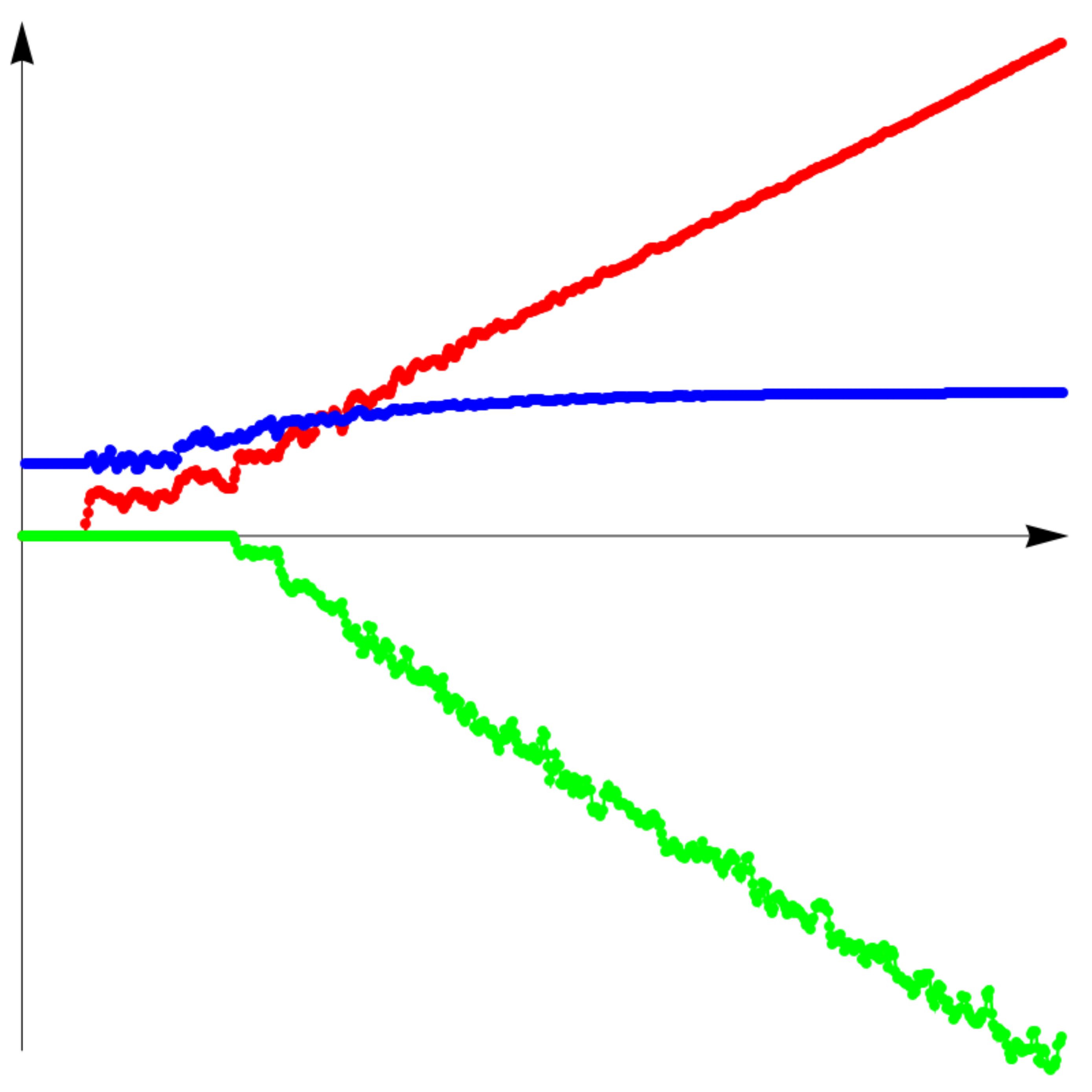}}
\caption{
The characteristic path length $\mu$ (red), the log of the cluster coefficient $\log(\nu)$ (green)
and the average vertex degree $\delta$ (blue) are shown together as a function of $n$ with logarithmic
scale in $n$; logarithmic growth therefore is shown linearly. First for Erdoes Renyi (where each edge is 
turned on with probability $p=0.1$), then for Watts-Strogatz for $k=4, p=0.1$ and then for 
Barabasi-Albert. In the second row we see first the case of two quadratic maps, then two exponential 
maps, and finally two random permutations.}
\end{figure}

An interesting quantity to study in our context is the {\bf length-clustering coefficient} 
\begin{equation}
 \lambda = \E[-\frac{\mu}{\log(\nu)}]   \; ,  \label{lengthcluster}
\end{equation}
where $\mu$ is the average characteristic length and $\nu$ the mean cluster coefficient. 
We are not aware that it has been studied already.  \\

The largest accumulation point limsup and the minimal accumulation point liminf can be called
the {\bf upper and lower Length-Cluster coefficient}. 
For most dynamical graphs, we see that the liminf and limsup exists for $n \to \infty$ and that 
the limit exists along primes. It exists also for Erdoes-Renyi graphs, where $\mu$ and $\nu$ have limits 
themselves also for fixed degree random graphs, where we see $\lambda$ to converge. The limit is 
infinite for Watts-Strogatz networks in the limit $n \to \infty$ due to large
clustering and zero for Barabasi-Albert networks in the limit $n \to \infty$. 
For graphs defined by random permutations, we see $\lambda$ converges to $1$.  \\

It has become custom in the literature to use a logarithmic scale in $n$ so that
one gets linear dependence of $\mu$ and $\log(\nu)$. Despite the fact that both clustering coefficient
and characteristic length are widely used, the relation between these two seems not have been 
considered already. While the clustering coefficient at a vertex is a local property which often
can be accessed, the characteristic length is harder to study theoretically. For random graphs
we measure clear convergence of $\lambda$ in the limit: \\

In the first case, the probability space is the set of all pairs of permutations on $Z_n$.
It is a set with $(n!)^2$ elements. Here are some questions. \\

\begin{center} \fbox{\parbox{12cm}{
{\bf Length-Cluster convergence conjecture  I:} for random permutation graphs with 2 or more
generators, the length-clustering coefficient (\ref{lengthcluster}) has a finite limit 
for $n \to \infty$. 
}} \end{center} 

In the next case, we take for $d=2$ the probability space $Z_n^2$ of all pairs $(a,b)$
leading to pairs of transformations $T(x)=x^2+a$, $S(x)=x^2+b$. Also for the following
question there is strong evidence: \\

\begin{center} \fbox{ \parbox{12cm}{ 
{\bf Length-Cluster convergence conjecture II:}
Graphs defined by $d \geq 2$ quadratic maps on $Z_p$ with prime $p$, 
the expectation of the limit $\lambda$ exists in the limit $p \to \infty$.
}} \end{center} 

\begin{figure}[H]
\scalebox{0.22}{\includegraphics{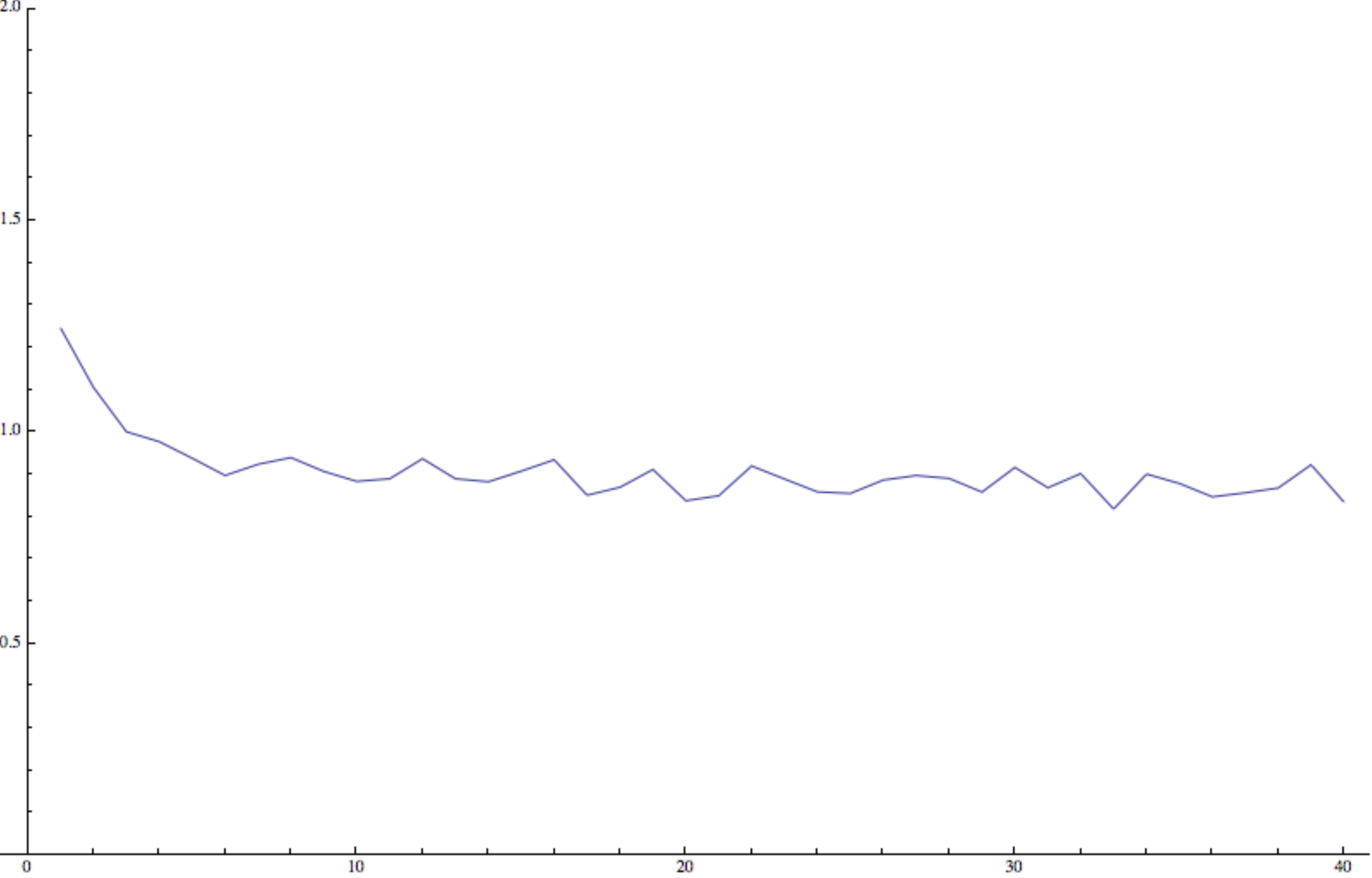}}
\caption{
This figure shows the expectation of the length-cluster coefficient for 
the graph generated by two random permutations $T,S$ on $Z_n$ as a function 
of $n$. For every $n$, we have computed the average of $500$ random permutations. 
Unlike for arithmetic maps, number theoretic considerations
play a lesser role for random permutations.}
\end{figure}

This remarkable relation between the 
global clustering coefficient $\nu$ which is the average of a local property
and the characteristic length $\mu$, which involves the average length of geodesics in the graph
and is not an average of local properties. 
Intuitively, such a relation is to be expected because a larger $\nu$ will allow for
shorter paths. If the limit exists then $\mu = -\lambda \log(\nu)$ which would be useful
to know because the characteristic length is difficult to compute while the clustering coefficient
$\nu$ is easier to determine. To allow an analogy from differential geometry, we can compare
the local clustering coefficient $\nu(x)$ with curvature, because a metric space large 
curvature has a small average distance between two points. \\

Relations between local properties of vertices and global characteristic length are not new.
In \cite{NewmanStrogatzWatts}, a heuristic estimate 
$mu=1+\log(n/d)/\log(d_2/d)$ is derived, where $d$ is the average degree
and $d_2$ the average $2$-nearest neighbors. Note that $d=2e/v$ so that $\log(n/d)$ can be
interpreted as the logarithm of the average edge density on $n$ vertices and that $\log(d_2/d)$
can be seen as a {\bf scalar curvature}. Unlike the Newman Strogatz Watts formula, which uses global
edge density we take the average of the edge density in spheres. To take an analogy of differential
geometry again, we could look at graphs with a given edge density and minimize the average path
length between two points. This can be seen as a path integral. We will look at the relation
of various functionals elsewhere.  \\

We have studied in \cite{GK2} graphs generated by finitely many maps 
$T_i(x) = [x^{\alpha_i} + a_i] \; {\rm mod} \; n$, 
where $a_i,n$ are positive integers and where $\alpha$ is a real parameter.
We see that for $\alpha$ larger than $1$ and not too close to $1$and $\alpha$ not
an integer, the graphs are essentially random, while for $\alpha<1$ or $\alpha$ close to $1$, 
there are geometric patterns. 
We are obviously interested in the $\alpha$ dependence. The reason for the interest
is that for $a_i=i$ we get Watts-Strogatz initial conditions for $p=0$. 
In \cite{GK2}, we especially took maps of the form 
$T_1(x) = x+1, T_i(x) = [x^{1+p} + i] \; {\rm mod} \; n, i \geq 2$
on $Z_n$, where $[x]$ denotes the floor function giving the largest integer smaller or 
equal to $x$. These are deterministic graphs which produce statistical properties as 
the Watts-Strogatz models. 

\begin{figure}
\scalebox{0.1}{\includegraphics{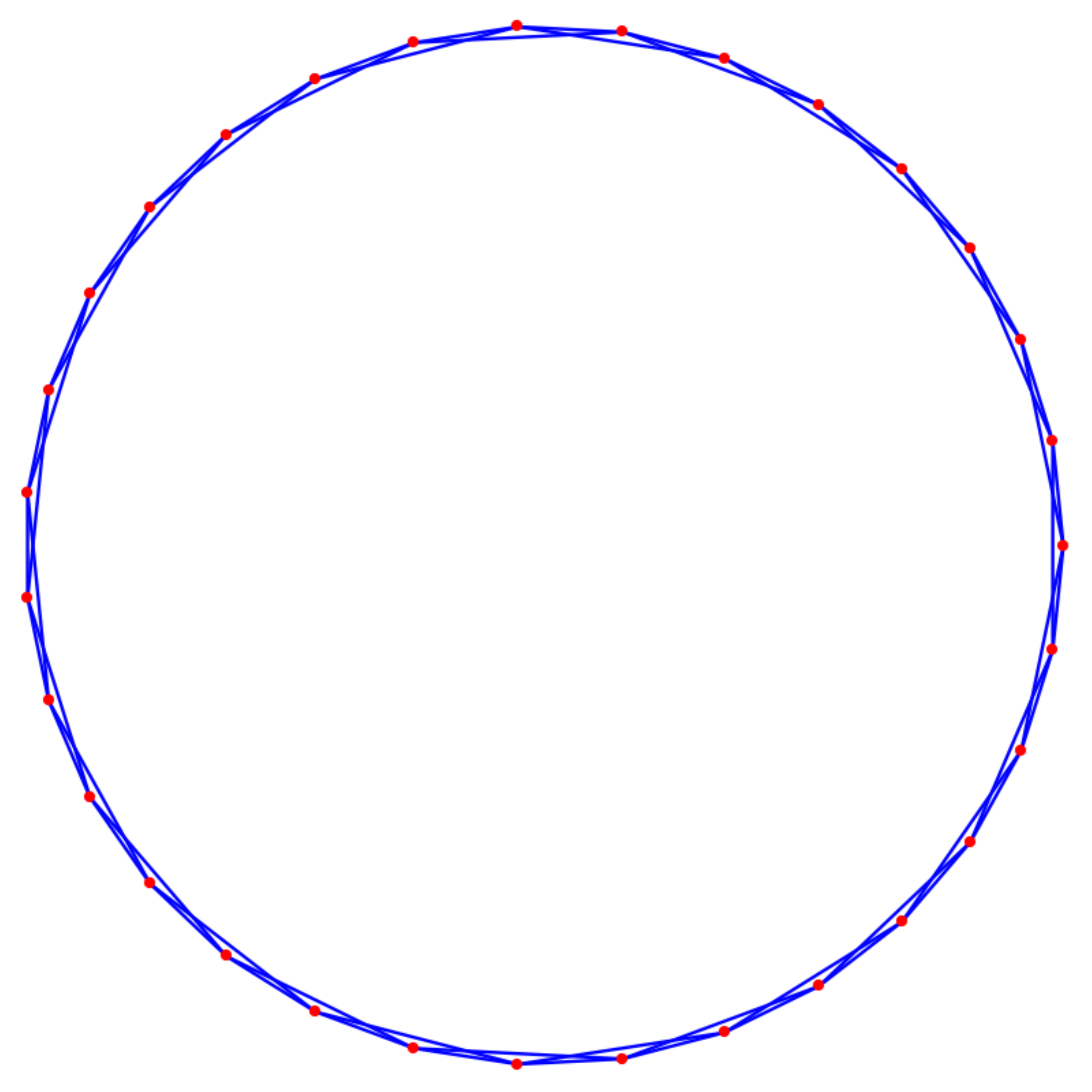}}
\scalebox{0.1}{\includegraphics{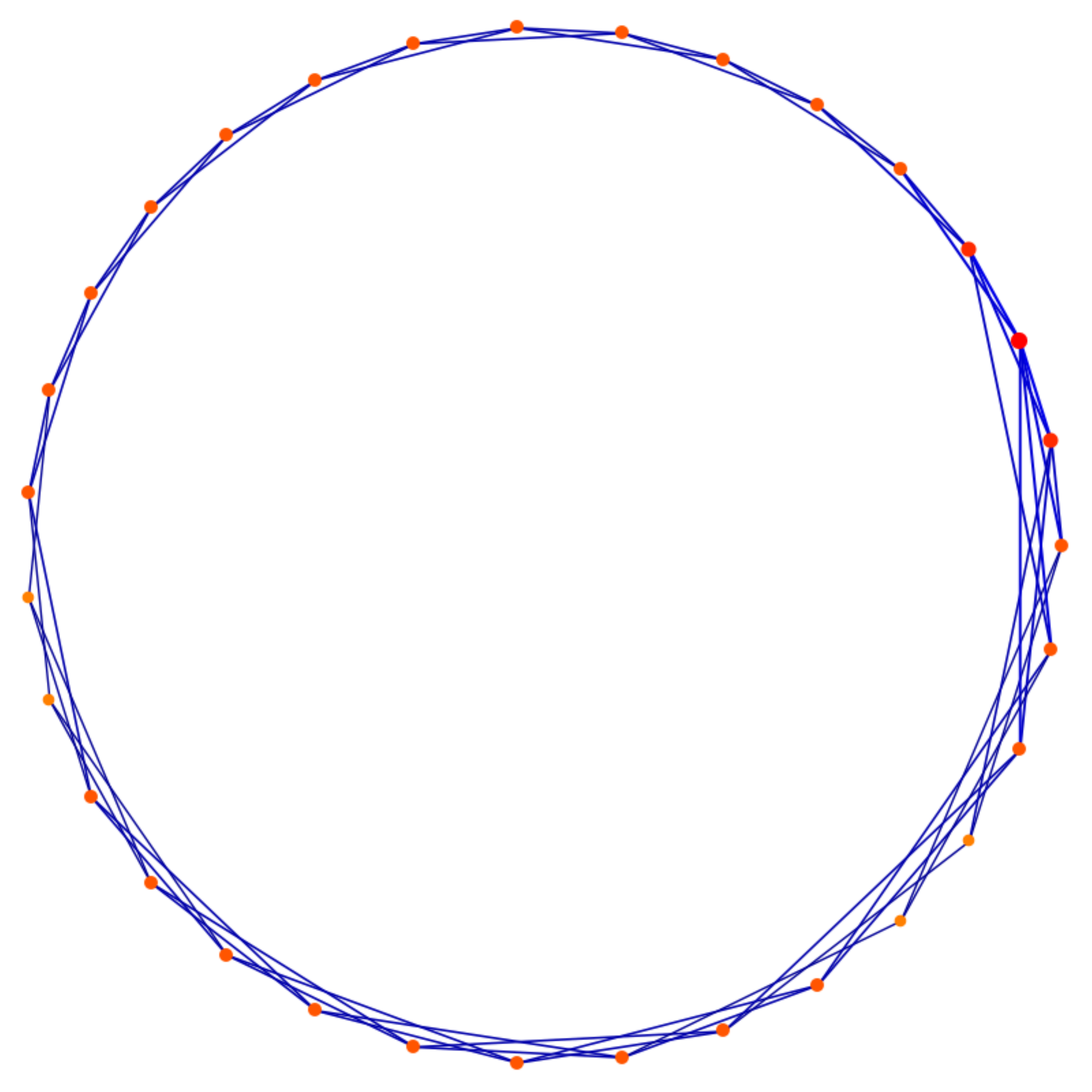}}
\scalebox{0.1}{\includegraphics{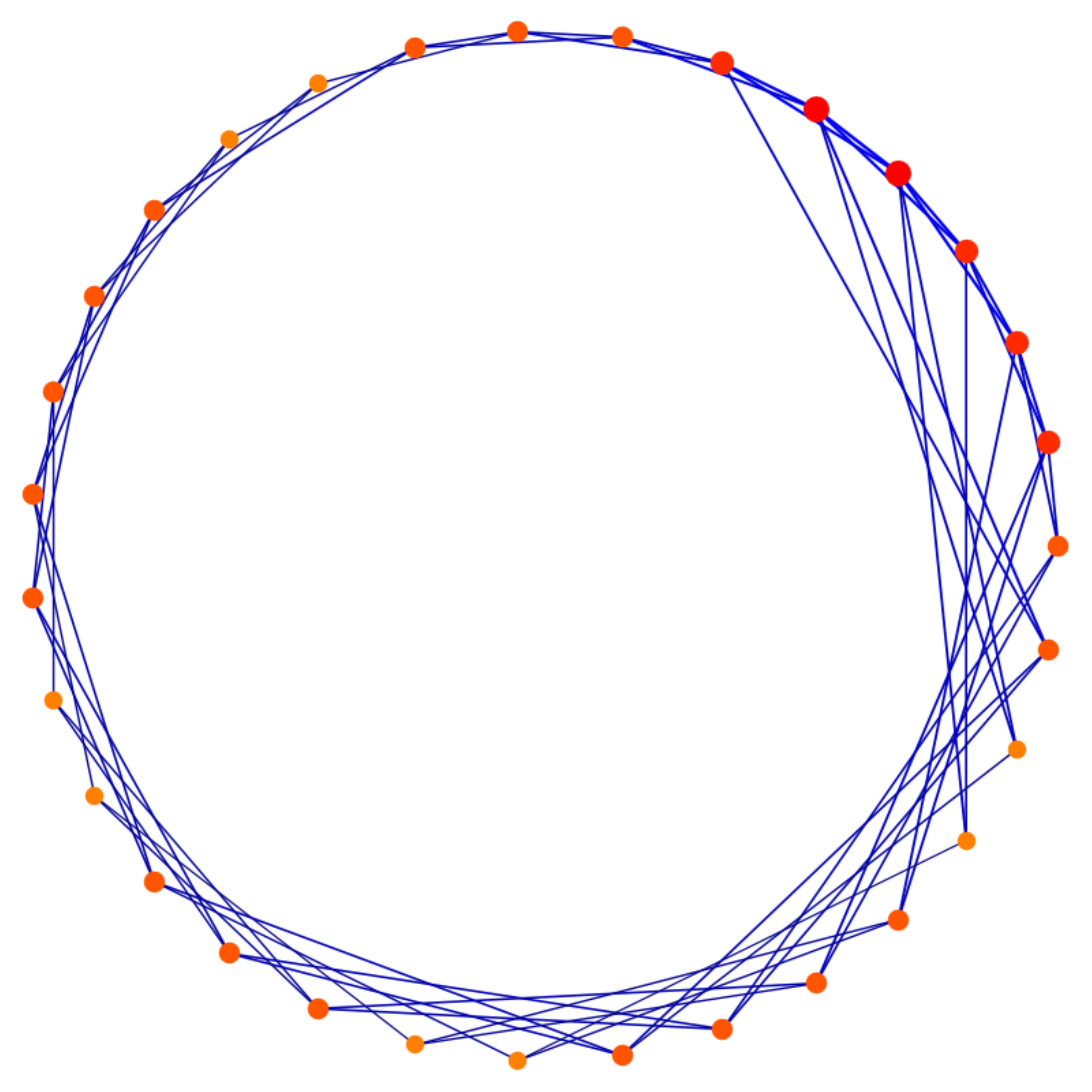}}
\scalebox{0.1}{\includegraphics{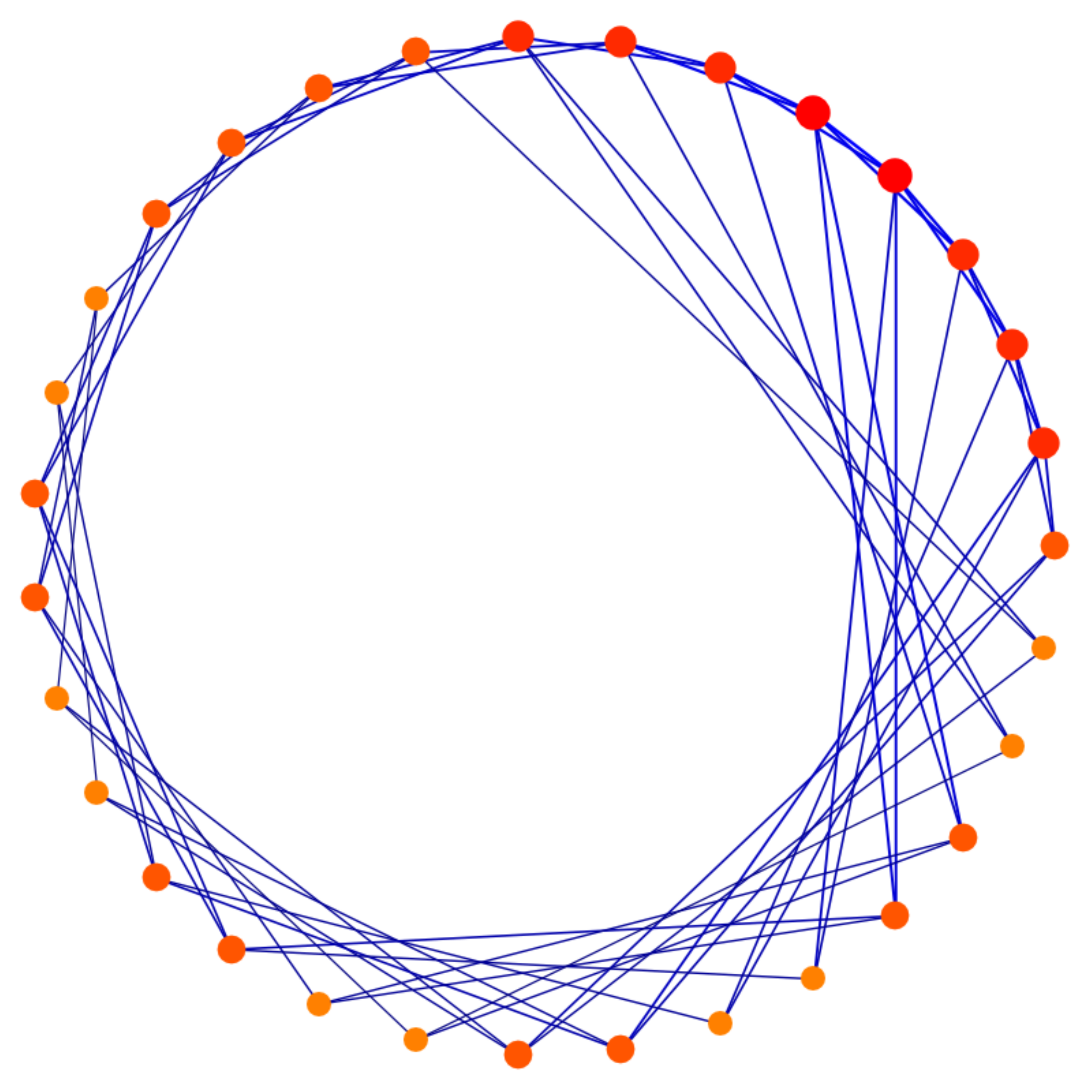}}
\caption{
The alpha graphs start like the Watts-Strogatz graphs. 
The rewiring is deterministic. See \cite{GK2}. 
}
\end{figure}

\section{Connectivity for quadratic orbital graphs}

We have looked already at special connectivity questions in \cite{KnillMiniatures}. 
One challenge for affine maps is to find necessary and sufficient
conditions that the map $T(x)=ax+b$ leads to a connected graph on $Z_n$.
An other {\bf network Mandelbrot challenge} is to that the graph on $Z_n$ 
generated by $T(x)=3x+1, S(x)=2x$ is connected for all $n$. We look at
other examples here, which are more of probabilistic nature. \\

Given a probability space of graphs, we can look at the
probability that a graph is connected. Here are three
challenges for quadratic maps along prime $n$: we see in one 
dimensions that connectivity gets rarer, in two dimensions that
it becomes more frequent and in three or higher dimensions that
connectivity is the rule. \\

In the following, we look at maps $T(x)=x^2+a$ generated on $Z_n$, 
where $n$ is prime. We denote by $C(n)$ the probability that the
graph is connected, where the probability space is is the set of $d$
different maps $T(x)=x^2+a$. Denote by $p_k$ the $k$'th prime. 

\begin{center} \fbox{\parbox{12cm}{
{\bf Quadratic graph connectivity A)}
With one quadratic map, the connectivity probability is
$$  C(p_k)  = O(p_k \frac{\log(k)}{k^2})   \; . $$
}} \end{center}

For random permutation graphs, we have $C(n)=1/n$ because there are
$(n-1)!$ cycles in a group of $n!$ transformations. Dixon's theorem
tells that the probability that two random transformations generate a 
transitive subgroup is $1-1/n+O(1/n^2)$ \cite{Cameron2011}.  \\

\begin{center} \fbox{\parbox{12cm}{
{\bf Quadratic graph connectivity B)}
If we have two quadratic maps, the connectivity probability is
$$ (1-C(p_k)) = O( (\frac{\log(p_k)}{k})^2) $$
}} \end{center}

\begin{center} \fbox{\parbox{12cm}{
{\bf Quadratic graph connectivity C)}
With three different quadratic maps on $Z_p$, then all graphs are connected. 
}} \end{center}

We have checked C) until prime $p=229$.

\section{Symmetries in arithmetic graphs}

If the maps under consideration preserve some symmetry, then also the graphs
can share this symmetry. We illustrate this with a simple example:

\begin{figure}
\scalebox{0.12}{\includegraphics{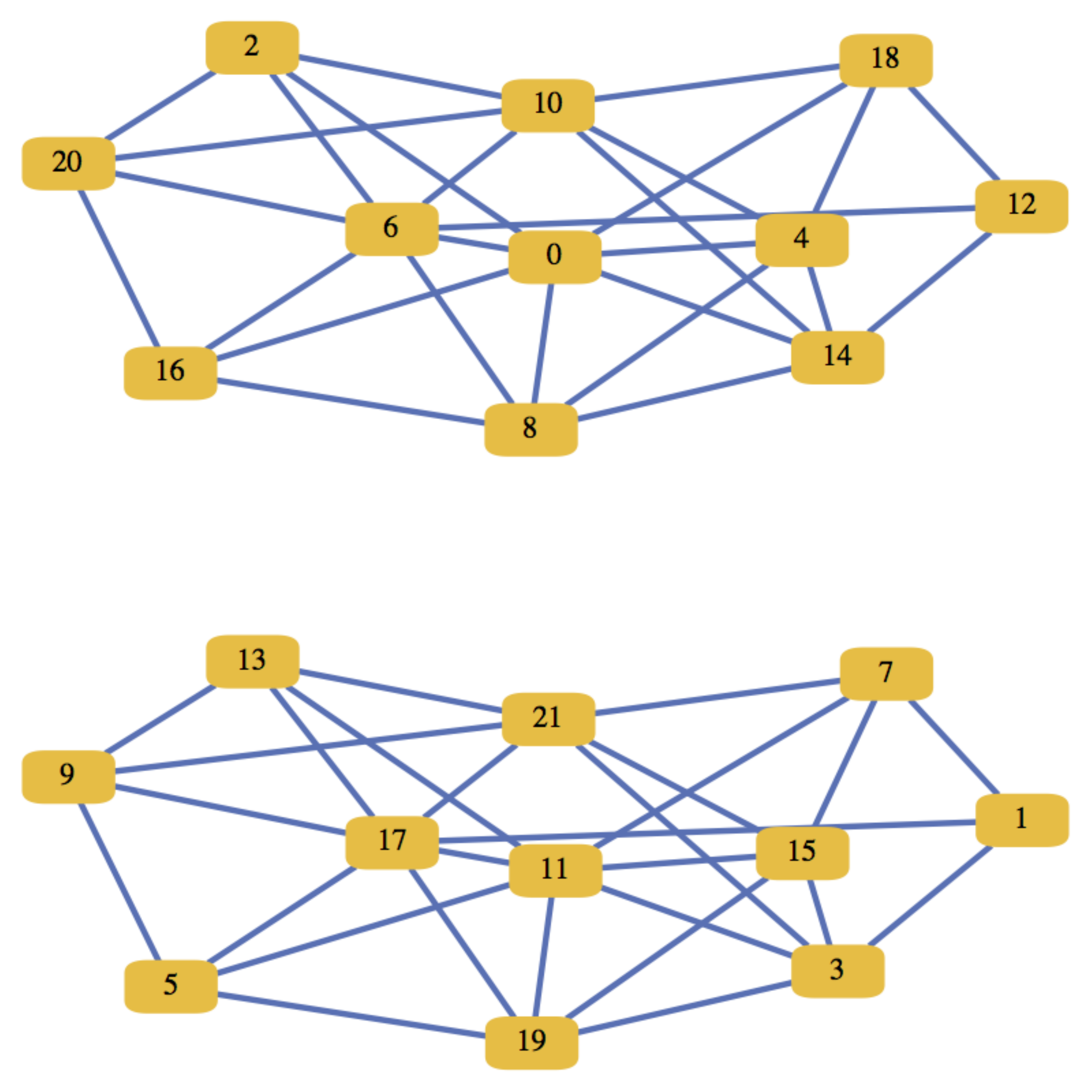}}
\scalebox{0.12}{\includegraphics{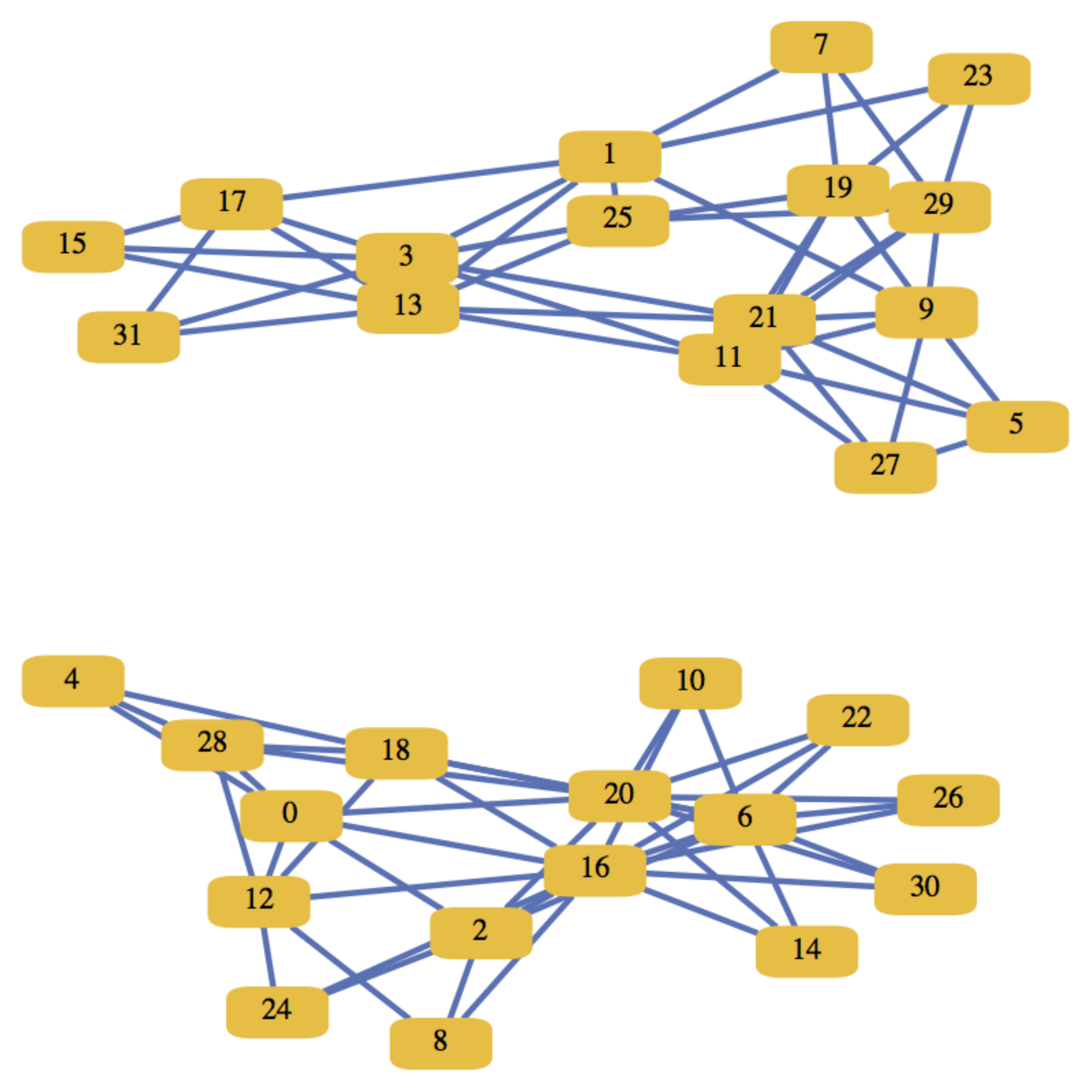}}
\scalebox{0.12}{\includegraphics{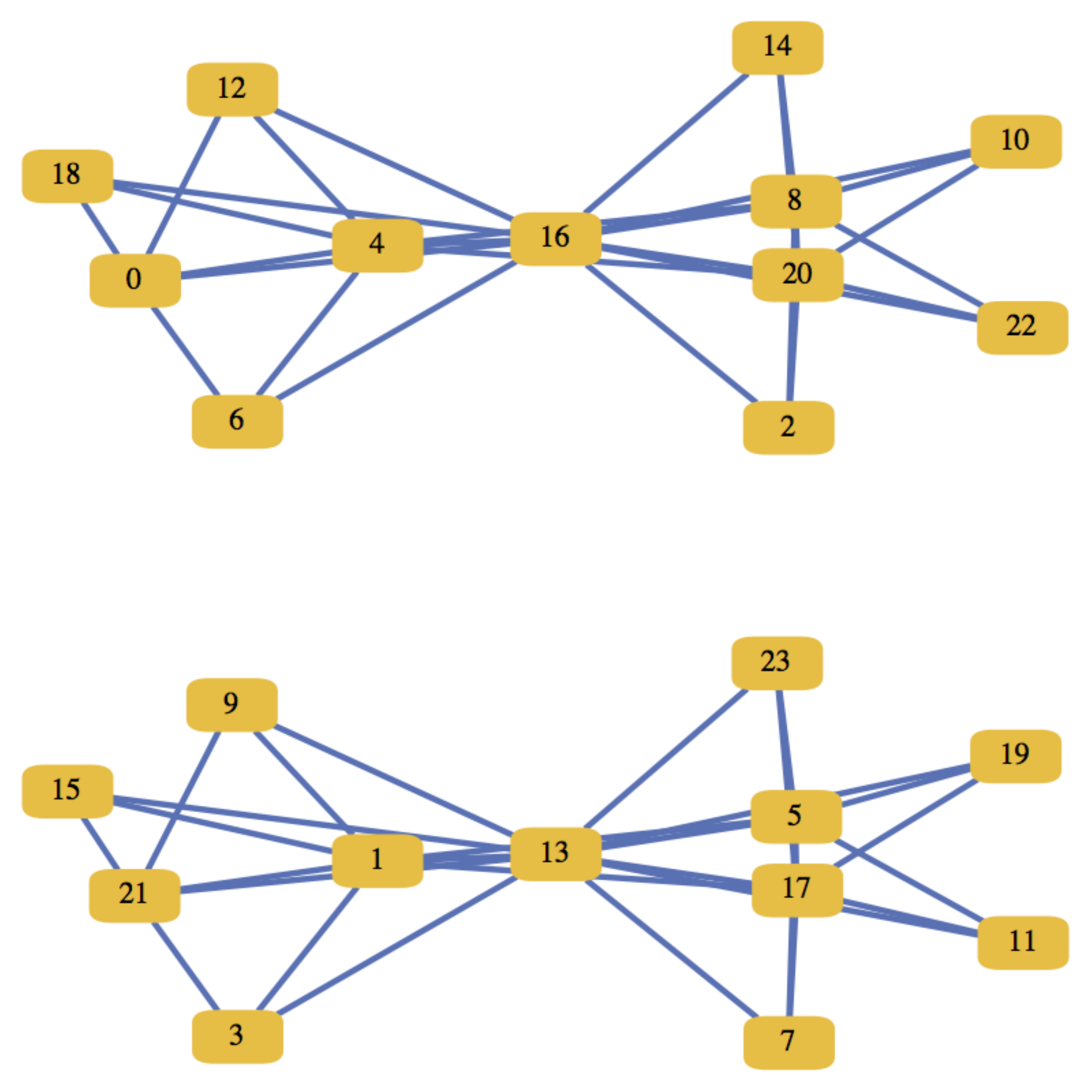}}
\caption{
The first figure shows the graph on $Z_{22}$ generated by
$f(x)=x^2+2,g(x)=x^2+6,h(x)=x^2+16$, where we have an isomorphism between
the two components. The second figure shows $n=32, f(x)=x^2+2,g(x)=x^2+12,h(x)=x^2+16$,
where we have no isomorphism. The third example shows
$n=24, f(x)=x^2+4,g(x)=x^2+12,h(x)=x^2+16$, one of the rare cases,
where we have an isomorphism, even so $n$ is a multiple of $8$. }
\end{figure}

\begin{figure}
\scalebox{0.22}{\includegraphics{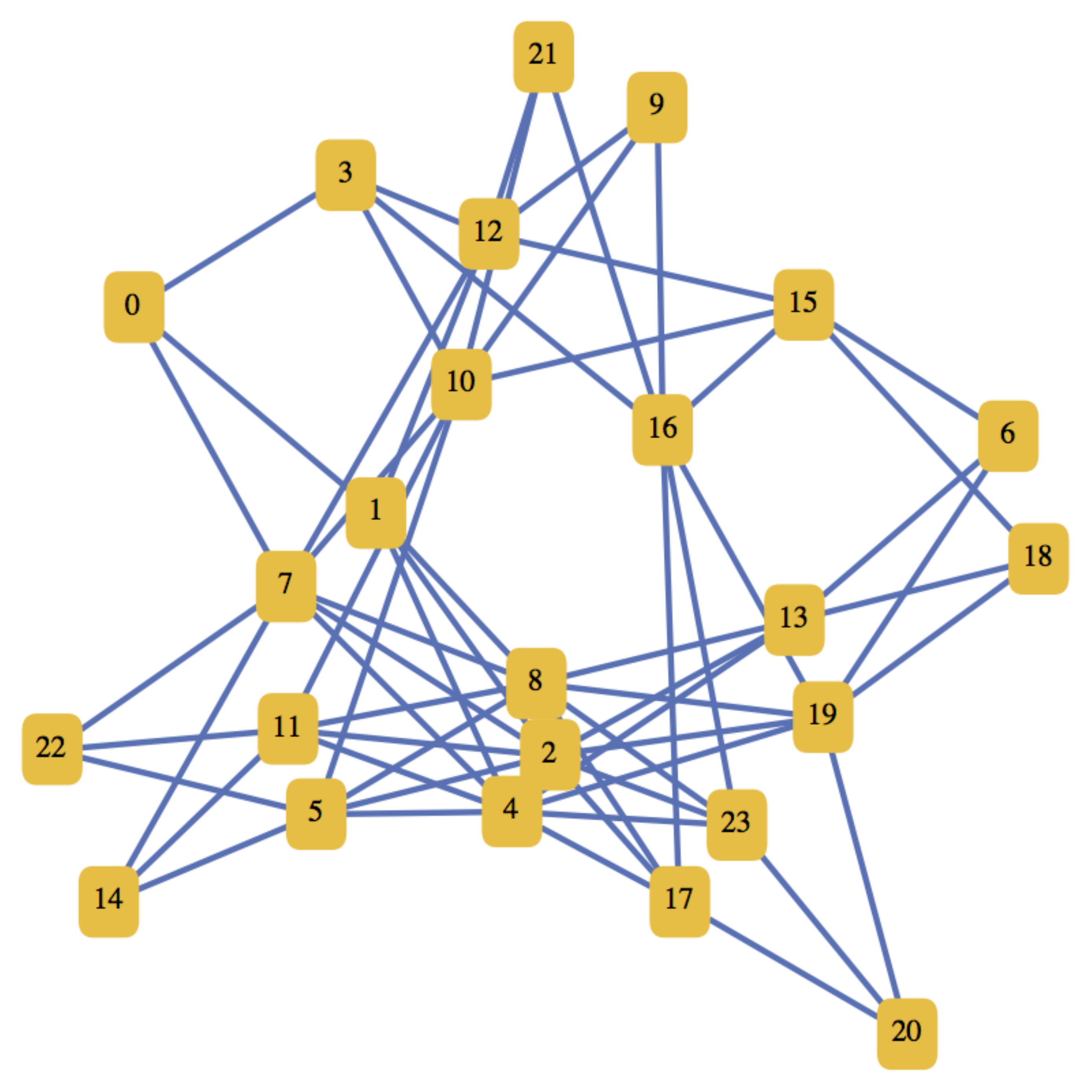}}
\scalebox{0.16}{\includegraphics{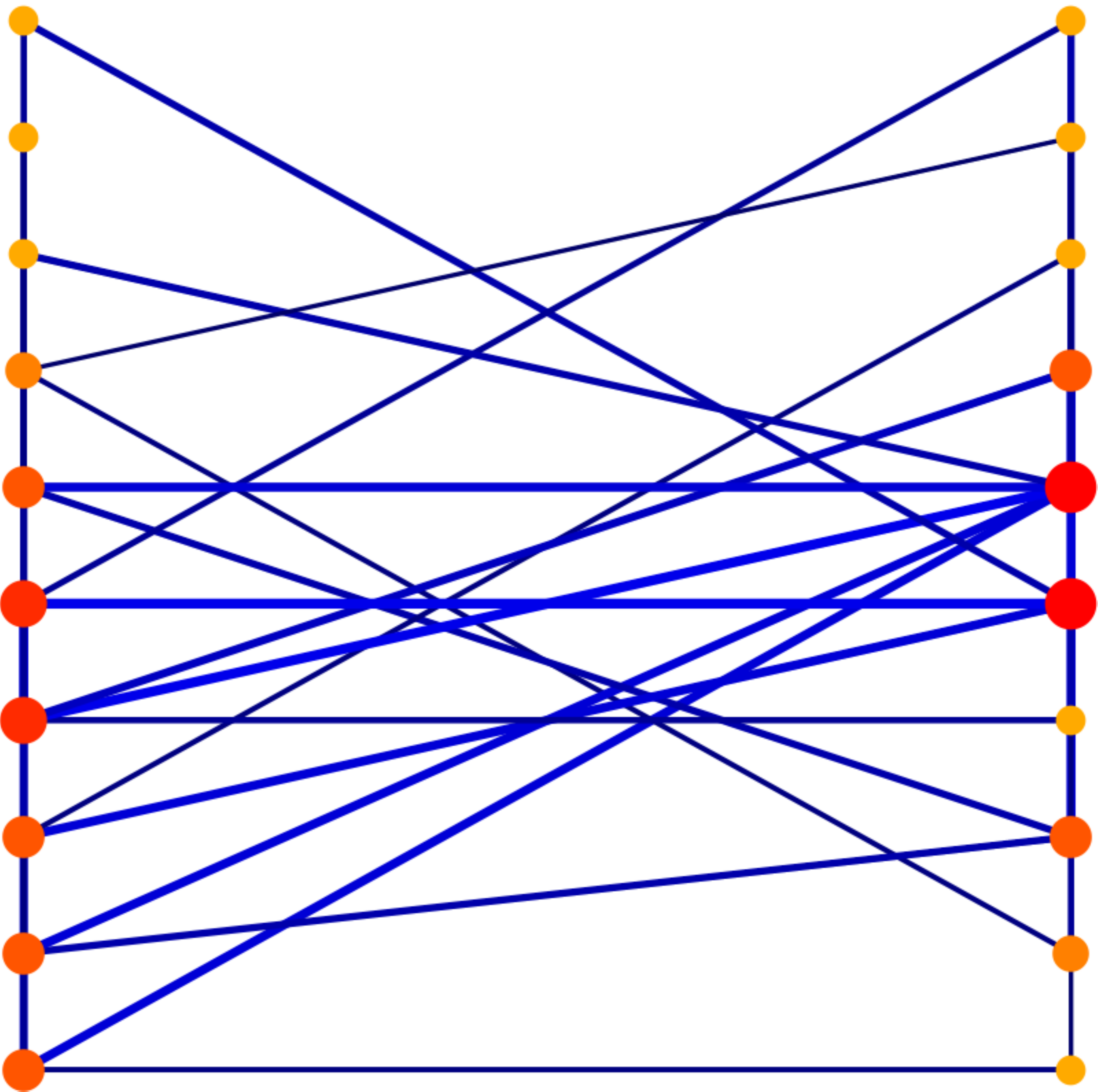}}
\caption{
The first figure shows the bipartite graph defined on $Z_{24}$ 
with  generators $x^2+1, x^2+3, x^2+7$.
The second graph shows it organized as a bipartite graph. 
}
\end{figure}

\begin{propo}[Miniature: symmetry]
Assume $n$ is even and $T_i(x) = x^2 + c_i$ on $Z_n$ and
$c_i$ are even, then the graph $G$ is the union of two disconnected graphs $G_1,G_2$. 
If $n$ is not a multiple of $8$, the two graphs $G_i$ are isomorphic. \\
\end{propo}

\begin{proof}
All maps leave the subsets of even and odd nodes invariant. This shows that the graph
is the union $G_1 \cup G_2$ of two graphs which have no connection. To see the 
isomorphism, we look at three cases, $n=8k+2,n=8k+4$ and $n=8kj+6$. 
In all cases, we construct an isomorphism which satisfies $\phi(T(x)) = T(\phi(x))$.\\
(i) If $n=8k+2$, then the isomorphism is $\phi(x) = x+n/2$. It
maps even nodes to odd nodes and vice versa. We check that
$$  \phi(T(x)) - T(\phi(x))  =  (8k+2) ( 2k+x ) $$
showing that the left hand side is zero modulo $n=8k+2$.
indeed, the left hand side is $(-4*k - 16*k^2 - 2*x - 8*k*x)$ which
agrees with the right hand side.  \\
(ii) If $n=8k+4$, then the isomorphism is $\phi(x) = x+n/4$. Also this
map gets from $G_1$ to $G_2$ and $G_2$ to $G_1$. Again, we check that
$$  \phi(T(x)) - T(\phi(x)) = (8k+4) (k+x)/2  \; . $$
Now, if $k$ is even, we apply this to the even nodes, if $k$ is odd
we apply this map to odd nodes. In goth cases, it is invertible.  \\
(iii) Finally, if $n=8k+6$, then the isomorphism is again $\phi(x) = x+n/2$.
Again we see that $\phi(T(x))-T(\phi(x))$ is a multiple of $n$ for any map $T$.
\end{proof}

\begin{propo}[Miniature: bipartite]
If $T_i(x) = x^2 + c_i$ on $Z_n$ where
$n$ is even and $c_i$ are odd, then the graph is bipartite and does not have any
odd cycles. It in particular does not have any triangles, so that the global
clustering coefficient is zero in that case.
\end{propo}
\begin{proof}
We can split the vertex set into an odd and even component $V=V_o+V_e$. 
Each of the maps $T_i$ forms a connection from $V_o$ to $V_e$.
\end{proof}

{\bf Remarks.}  \\
{\bf 1)} We see experimentally that if $n$ a multiple of $8$, then for most cases,
there is no isomorphism of the two component. We expect the
probability of such events goes to zero for $n \to \infty$. Statistically, it is
unlikely that two random graphs are isomorphic, so that we just might have rare events.  \\
{\bf 2)} If $n$ is even and not divisible by $4$ and all generators are polynomials which
preserve odd or even numbers, then the graph $G$ is the disjoint union of
two graphs. If all generators switch from odd to even numbers, then the
graph is bipartite. 
If $\phi(x) = x+n/2$ and $n=4k+2$ and $\phi(T(x))- T(\phi(x))$ is a multiple of $n$
then $\phi$ is a conjugation. In any case, if all $T_i$ switch from even to
odd, then the two parts of the bipartite graph are the even and odd numbers.

\section{Remarks} 

Here are some additional remarks: \\

{\bf 1)} Motivated by the Mandelbrot set which is defined as the set
the parameters $c=a+ib$ for which the Julia set $J_c$ is not connected, we can look at all
the parameter $a,b$ for which the arithmetic graph generated by 
$x^2+a,x^2+b$ acting on $R=Z_n$ is connected.
This is encoded in the matrix $A_{ab}$ which gives the number of 
components of the graph. It depends very much on number theoretical properties. 

{\bf 2)} How many quadratic generators are necessary on $Z_n$ to reach a 
certain edge density? Certainly, $n$ generators $T_i(x) = x^2+i$ suffice to 
generate the complete graph $K_n$. The edge density $|E|/|V|$ is half of the 
average degree $d$. By adding more generators, we increase $d$ and $C$
and decrease $L$. It is in general an interesting {\bf modeling question}
to find a set of polynomial maps which produce a graph similar to a given 
network.  

{\bf 3)} On $R=Z_n$ with $d$ generators of the form $x^2+a_i$, we call
$r_d(n)$ the smallest diameter of a graph which can be achieved. 
How fast does $r_d(n)$ grow for $n \to \infty$? This is already interesting for $d=1$,
where we see not all minimal diameters realized. For $n=90$ and $d=1$ for example,
all graphs are disconnected so that $r_1(90)=\infty$. 
For $n=466$, the minimal diameter is $55$, the next record is $r_1(486)=85$,
then $r_1(1082) = 89, r_1(1454)=93$ and $r_1(1458)=247$. We did not find 
any larger minimal diameter for $n<2000$.  
For $d=2$, the diameters are smaller. For $n=2$ the minimal diameter is $1$,
for $n=4$, it is $2$, for $n=9$ it is $3$, for $n=17$ it is $4$, for $n=30$ it is $5$
for $n=67$ it is $6$. For $n=131$ the minimal diameter is $7$. For $n=233$ we reach $8$. 
We still have to find an $n$ where the minimal diameter is larger than $8$. 
For $d=3$, where we average over $n^3$ graphs, we see 
$r_3(16)=3$ and $r_3(41)=4$ and $r_3(97)=5$. 
We did not find $n$ yet for which $r_3(n)>5$. 

{\bf 4)} Let $c_d(k)$ denote the largest clique size which can be achieved 
by $d$ quadratic maps.  If the monoid $R$ is generated by two generators, 
then the graph can have $K_4$ subgraphs but it is very unlikely. 
For three generators, it happens quite often. For example, for $n=40$
and $f(x)=x^2+4,g(x)=x^2+29, h(x)=x^2+24$ the graph $G$ has $9$ cliques $K_4$. 
The question is already interesting for $d=1$, where we the maximal cliques are triangles. 
Triangles are rare but they occur. The graph generated on $Z_{57}$ with $f(x)=x^2+30$ has 
two  triangles.  This is equivalent to the fact that the Diophantine 
equation $f(f(f(x)))=x$ of degree $8$ has a solution modulo $57$. 

{\bf 5)} Many graphs on a ring $R$ constructed with maps in $T$ have symmetries. 
If all elements $T$ are invertible, then we have a group action on $R$ and this group 
is a subgroup of the automorphism group of the graph. 
We can for example take $n$ prime and $T=\{ f,g \}$ acting on $Z_n^2$ by two Henon type maps
$f(x,y) = (x^2-y,x), g(x,y) = (y,y^2-x)$. The subgroup of permutations of $R$ generated by these two
permutations is a subgroup of the automorphism group of $G$. 

{\bf 6)} We can compute the global clustering coefficient in the graph as an 
expectation $C(n)$, when looking 
at the probability space of all pairs of quadratic maps $\{f(x) = x^2+a, g(x) = x^2+b \}$. We measure
the average to decay like $3/n$. This means that we expect in $3$ of $n$ vertices to have a triangle.
A triangle means either that $T^2(x)=S(x)$ or $T(x)=S^2(x)$. Counting the number 
of solutions to the Diophantine equations $(x^2+a)^2+a = x^2+b$ which is 
$x^4 + x^2(2a-1) + (a + a^2+b) =0$ modulo $n$. The measurements show that when taking
$a,b$ random, we have $3$ solutions in average. We have $4$ solutions in general but
if $n$ factors, then there can be more. On the other hand,
multiple solutions brings the average down.  For three maps $d=3$, we see that the 
clustering coefficients decays like $4/n$ and the average degree to be close to $6$. 
When comparing with random graphs, where the average degree is $p$, we see that the 
clustering coefficient is proportional to $p$.  

{\bf 7)} The average degree of a graph $G=(V,E)$ is $a_2(n)=2|E|/|V|$ by the Euler handshaking lemma.
In our case, the average vertex degree for two quadratic polynomials fluctuates but converges
to $2d$, where $d$ is the number of generators. 
We have computed it for all polynomial pairs on $Z_n$ for $n=1$ to $n=200$ and $d=2$. 
The average vertex degree is exactly $4$ if all $T_1(x) \neq T_2(x)$. 
The difference $a_2(n)-4$ depends on the number of solution pairs $x^2+a=y^2+b$ and
is of number theoretical nature. The local maxima are obtained if $n$ is prime.

{\bf 8)} The characteristic path length is defined as the average distance between two 
vertices in the graph. There are not many analytical results available (see \cite{FFH}). 
The networks generated by two quadratic maps have a
characteristic path length which grows logarithmically, similar than Watts-Strogatz. 
We can slow it down and behave like Barabasi-Albert if we take generators $f,g,fg,gf$
which naturally also brings the clustering coefficient up. An other possibility is to 
add an affine map $T(x)=x+1$ (see \cite{GT2}). 

{\bf 9)} The vertex distribution can depend very much on arithmetic properties. If $n$ 
is prime and $T$ consists of two different quadratic maps $x^2+a$ and $x^2+b$
then only vertex degrees $2,3,4,5,6$ can appear, half have degree $2$ or $6$ and 
half have degree $4$. We can get smoother vertex degree distributions by taking maps like 
$f,g,f(g),g(f),f(f),g(g)$. 

{\bf 10)} If $T$ is generated by a single map, the dimension is $\leq 2$
and the Euler characteristic of the graph is nonnegative. The reason is that 
there are no $K_4$ subgraphs and so no $K_5$ subgraphs. 
We see no $K_{3,3}$ subgraphs indicating that all these graphs are planar.  
In any case, the Euler characteristic is $\chi(G) = c_0-c_1+c_2$ because $c_n=0$ for $n \geq 3$.
By the Euler Poincar\'e formula it is $\chi(G)=b_0-b_1$, where $b_0$ is the number of components and
$b_1$ is the number of cycles $b_1$. Because every attractor is homotopic to a cycle because every
orbit eventually loops on a cycle, the number of components is larger or equal than the 
number of cycles. More generally, the dimension of an arithmetic graph generated by $d$ 
transformations has dimension $\leq d+1$. It would be interesting to get bounds on the 
Euler characteristic.  

{\bf 11)} Instead of $Z_n$ we can take rings like a finite ring $Z_n[i]$ of 
Gaussian integers. One could ask, for which $n$ the graph on the ring of $Z_n[i]$ generated by 
$x^2,x^3,x^5$ is connected. 

\vspace{12pt}
\bibliographystyle{plain}

\end{document}